\documentclass{article}
\usepackage{a4wide}
\usepackage{CJKutf8}
\usepackage{amsmath,amsfonts, bm}
\usepackage{amsthm}
\usepackage{amscd}
\usepackage{amssymb}
\usepackage{mathrsfs}
\usepackage{mathtools}
\usepackage{mathdots}
\usepackage{tikz}
\usetikzlibrary{arrows.meta, automata, chains,bbox, positioning, quotes}
\usepackage[all]{xy}
\usepackage{graphicx} 
\usepackage{xcolor}
\usepackage[colorlinks=true, pdfstartview=FitV, linkcolor=purple, citecolor=purple,pagebackref]{hyperref}
\usepackage{color}
\usepackage{cleveref}
\usepackage{bbm}
\usepackage{float}
\usepackage{import}
\usepackage{enumitem}
\usepackage{listings}
\usepackage{lineno}
\usepackage{bbold}
\usepackage{tabularx}
\usepackage{multirow}
\usepackage[toc,page]{appendix}

\newif\ifdraft
\draftfalse  
\ifdraft
  \newcommand{\SB}[1]{\textcolor{blue}{#1}}
  \newcommand{\SR}[1]{\textcolor{red}{#1}}
  
\else
  \newcommand{\SB}[1]{#1}
  \newcommand{\SR}[1]{#1}
  
\fi

\renewcommand{\textcolor}[2]{#2}

\DeclareSymbolFont{lettersA}{U}{txmia}{m}{it}
\DeclareMathSymbol{\Indi}{\mathord}{lettersA}{'211}
\usepackage{comment}

\newtheorem{Def}{\bf Definition}
\newtheorem{Th}{\bf Theorem}
\newtheorem{Lem}{\bf Lemma}
\newtheorem{Prop}{\bf Proposition}
\newtheorem{Rem}{\bf Remark}
\newtheorem{Ex}{\bf Example}

\newtheorem{Cor}{\bf Corollary}

\newcommand{\R}{{\mathbb R}}

\newcommand{\Z}{{\mathbb Z}}

\newcommand{\N}{{\mathbb N}}
\newcommand{\A}{{\mathcal A}}
\newcommand{\Ab}{{A}}

\newcommand{\eps}{\varepsilon}

\usepackage{todonotes}
\title{The natural extension of the $(-\beta)$-transformation}
\author{Shigeki Akiyama, Hiromi Ei, and Hiroaki Ito}
\date{}

\begin{document}
\maketitle
\begin{abstract}
We give a concrete construction of a natural extension of the $(-\beta)$-transformation when $\beta$ is greater than the golden mean. 
Our construction relies on 
its Markov diagram and the eigenvectors of the associated countable Markov shifts. Its positive recurrence 
can be shown by path counting using a special property of the diagram.
Our down-to-earth construction elucidates the result of Bruin-Kalle \cite{BruinKalle} by examples.
\end{abstract}

\section{Comparison between $\beta$- and $(-\beta)$-transformations}\label{Basic properties of (-beta)-transformation}
Let $\beta$ be a real number greater than $1$ and $\A=\{0,1,\dots, 
\lfloor \beta \rfloor\}$.
The $\beta$-transformation is an interval map $T_\beta$ on $[0,1)$ 
defined by
$$
T_\beta:x \longmapsto \beta x -\left\lfloor \beta x\right\rfloor.
$$
By iterating this map $T_\beta$, we obtain an expansion
$$
x=\sum_{i=1}^{\infty}\frac{d_i}{\beta^i}=(d_1d_2\dots d_i\cdots)_\beta,
$$
where $d_i=d_i(x):=\lfloor \beta T_{\beta}^{i-1}(x)\rfloor\in \A$.
\textcolor{red}{
The $\beta$-expansion of $x$ is denoted by $d(x,\beta)=d_1(x) d_2(x) \cdots$.
Define 
\[
d^*(1,\beta) :=\lim_{x\nearrow 1} d(x,\beta)= b_1 b_2 \cdots
\]}
in the topology of sequence space.
It is well known that 
$
d^*(1,\beta) 
$
plays a crucial role in describing the dynamics of $\beta$-expansion. \textcolor{red}{Parry \cite{Parry:60} and Ito-Takahashi \cite{Ito-Takahashi:74} showed that an infinite word $x_1x_2\dots\in \A^{\N}$ is realized as $d(x,\beta)$ for some $x\in [0,1)$
if and only if 
$$x_nx_{n+1}\dots \ll d^*(1,\beta)$$ for all $n\in \N$.
Here $\ll$ is the lexicographic order.}
Parry \cite{Parry:60} also 
gave a combinatorial characterization of $d^*(1,\beta)$. In the notation of \cite{Ito-Takahashi:74}, Corollary 1 in \cite{Parry:60} can be stated as follows: a sequence $(b_i)_{i\ge1}\in \A^{\N}$ arises as $d^*(1,\beta)$ for some $\beta>1$ if and only if
\begin{equation}
\label{SelfBeta}
b_nb_{n+1}b_{n+2}\dots \ll b_1b_{2}b_{3}\dots 
\end{equation}
for all $n\ge 1$. 
A natural extension of $\beta$-expansion is constructed in \cite[p.104-]{Dajani-Kraaikamp02} and \cite{Dajani-Kraaikamp-Solomyak96}.
Let
\begin{align*}
\mathcal{R}_{\beta,i}=\left[0, T_\beta^i(1)\right]\times \left[0, \frac{1}{\beta^i}\right]
\end{align*}
for all $i\ge 0$ and the underlying space $\mathcal H_\beta$ is obtained by 
stacking $\mathcal R_{\beta,i+1}$ on the top of $\mathcal R_{\beta,i}$ for each $i \ge 0$. 

If $(x,y)\in\mathcal{R}_{\beta,i}$, let $x=(d_1d_2\cdots)_{\beta}$, $y=(\underbrace{0\cdots0}_{i}c_{i+1}c_{i+2}\cdots)_{\beta}$. Then, 
\begin{align}
\label{NaturalBeta}
\mathcal{T}_{\beta} (x,y):= \left(T_{\beta}x, y^*\right)\in
\begin{split}
\left\{
\begin{array}{ll}
\displaystyle  \mathcal{R}_{\beta,0}, \quad &\text{if}~ d_1<b_{i+1}, \\
[5pt]
\displaystyle \mathcal{R}_{\beta,i+1}, \quad & \text{if} ~ d_1=b_{i+1},
\end{array} \right.
\end{split} 
\end{align}
where
\begin{align*}
y^{*}=
\begin{split}
\left\{
\begin{array}{ll}
\displaystyle \frac{b_1}{\beta}+\cdots+\frac{b_i}{\beta^i}+\frac{d_1}{\beta^{i+1}}+\frac{y}{\beta}=(b_1\cdots b_id_1c_{i+1}c_{i+2}\cdots)_\beta, \quad & \text{if}~ d_1<b_{i+1}, \\
[10pt]
\displaystyle \frac{y}{\beta}=(\underbrace{0\cdots0}_{i+1}c_{i+1}c_{i+2}\cdots)_\beta, \quad & \text{if} ~ d_1=b_{i+1}. 
\end{array} \right.
\end{split}
\end{align*}

An example of this construction is depicted in \Cref{Fig:NE_positive}. In this manner, the left sides of the rectangles are aligned, and the thicknesses of the rectangles are monotonically decreasing.
Its absolutely continuous invariant density (not normalized) is
given as a projection of the invariant measure (2-dimensional Lebesgue measure) 
of this natural extension:
$$
h_{\beta}(x)= \sum_{n\ge 0,~ x< T_{\beta}^n(1)}\frac{1}{\beta^n}.
$$

\begin{figure}[ht]
\centering
  \includegraphics[width=8.5cm]{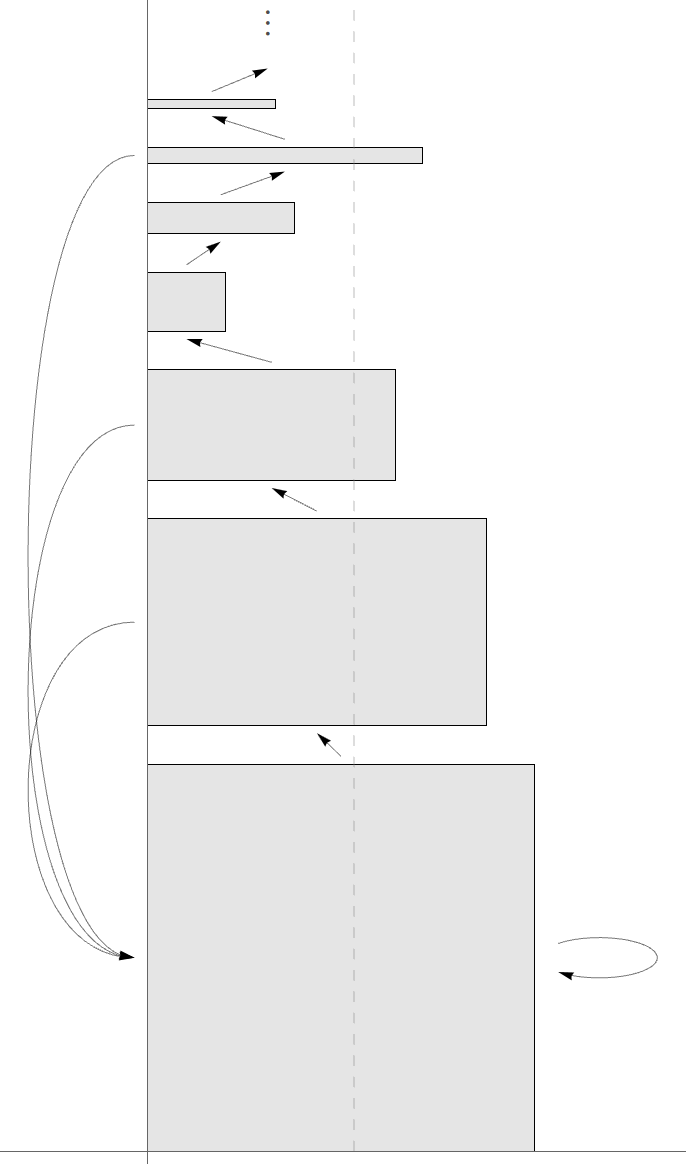} 
\caption{Natural Extension of the $\beta$-transformation for $d^*(1,\beta)=111001\sum_{k=1}^{\infty}0^k1$. (The dashed line is $x=1/\beta$.)}
\label{Fig:NE_positive}
\end{figure}

Ito-Sadahiro \cite{Ito-Sadahiro09} denoted by $I_\beta$ the half-open interval $[\ell_\beta, r_\beta)=[-\frac{\beta}{\beta+1}, \frac{1}{\beta+1})$ and defined $(-\beta)$-transformation on $I_\beta$ by 
$$
T_{-\beta}:x \longmapsto-\beta x -\left\lfloor -\beta x-\ell_\beta\right\rfloor.
$$
Then, for each $x\in I_\beta$, we have a $(-\beta)$-expansion $x=(d_1d_2\cdots)_{-\beta}$, where 
\begin{align*}
d_i=d_i(x):=\left\lfloor -\beta T_{-\beta}^{i-1}(x)-\ell_\beta \right\rfloor\in \A,
\end{align*}
and $\A = \{0,1,\ldots, \lfloor \beta \rfloor \}$. 
\textcolor{blue}{
By applying the theorem of Li and Yorke \cite{LiYorke}, it is known that
$T_{-\beta}$ has a unique invariant measure absolutely continuous with respect to the Lebesgue measure and hence is ergodic.}
The following invariant density is due to \cite{Ito-Sadahiro09}:
\begin{equation}
\label{Density}
h_{-\beta}(x)= \sum_{n\ge 0,~ T_{-\beta}^n(\ell_\beta)\le x}^{\infty}\frac{1}{(-\beta)^n}
\end{equation}
which is a solution of Kuzmin's equation. However, its negative terms make it difficult to obtain a clear understanding of its dynamical system. 

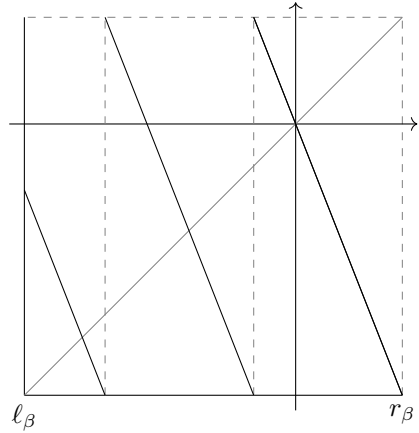
\begin{figure}[htbp]
\centering
\begin{tikzpicture}

\pgfmathsetmacro{\B}{2.543}
\pgfmathsetmacro{\L}{-5*\B/(\B+1)}
\pgfmathsetmacro{\R}{5/(\B+1)}

\draw[thin] (\R,\L) -- (\L,\L) node[below] {$\ell_\beta$};
\draw[thin] (\L,\L) -- (\L,\R);
\draw[very thin,gray,dashed] (\R,\R) -- (\L,\R);
\draw[very thin,gray,dashed] (\R,\L) -- (\R,\R);
\draw[->] ({\L-0.2},0) -- ({\R+0.2},0);
\draw[->] (0,{\L-0.2}) -- (0,{\R+0.2});

\draw[very thin,gray,dashed] (\R-5/\B,\L) -- (\R-5/\B,\R);
\draw[very thin,gray,dashed] (\R-10/\B,\L) -- (\R-10/\B,\R);

\pgfmathsetmacro{\Ra}{\R-5/\B}
\pgfmathsetmacro{\Rb}{\R-10/\B}

\draw[domain=\L:\R, very thin, gray] plot(\x,\x);
\draw[domain=\R-5/\B:\R] plot(\x,{-\B*\x}) node[below] {$r_\beta$};
\draw[domain=\Ra:\R] plot(\x,{-\B*\x});
\draw[domain=\Rb:\Ra] plot(\x,{-\B*\x-5});
\draw[domain=\L:\Rb] plot(\x,{-\B*\x-10});
\end{tikzpicture}
\caption{The graph of $T_{-\beta}$ $(\beta=2.54\dots)$}
\label{graph-T-beta}
\end{figure}

In this paper, we give a construction of the natural extension in this $(-\beta)$-expansion only using positive terms, analogous to (\ref{NaturalBeta}).
Bruin-Kalle \cite{BruinKalle} gave a general approach to this type of piecewise linear maps, assuming several axioms. We revisit this in the setting of $(-\beta)$-expansion and give a down-to-earth construction. 
\textcolor{red}{The main technical difficulty is to show that the corresponding countable Markov shift is positive recurrent. We directly prove this in \Cref{SystemPositiveRecurrent}
by the path counting method using a particular property of the associated Markov diagram 
in Corollary \ref{Defect}.}
The reader will notice that the countable Markov shift gives a concrete way for construction, i.e., one can determine the rectangles explicitly as in \Cref{Th:left and right eigenvector}, see \Cref{Fig:NE_Example1and2}. 
We observe that the natural extension is more involved than that of $\beta$-expansion: the thicknesses of the rectangles are not monotone decreasing, and the sides of the rectangles are not aligned, compare \Cref{Fig:NE_positive}.

\begin{figure}[p]
\centering
  \includegraphics[width=\textwidth,
  height=\textheight,
  keepaspectratio]{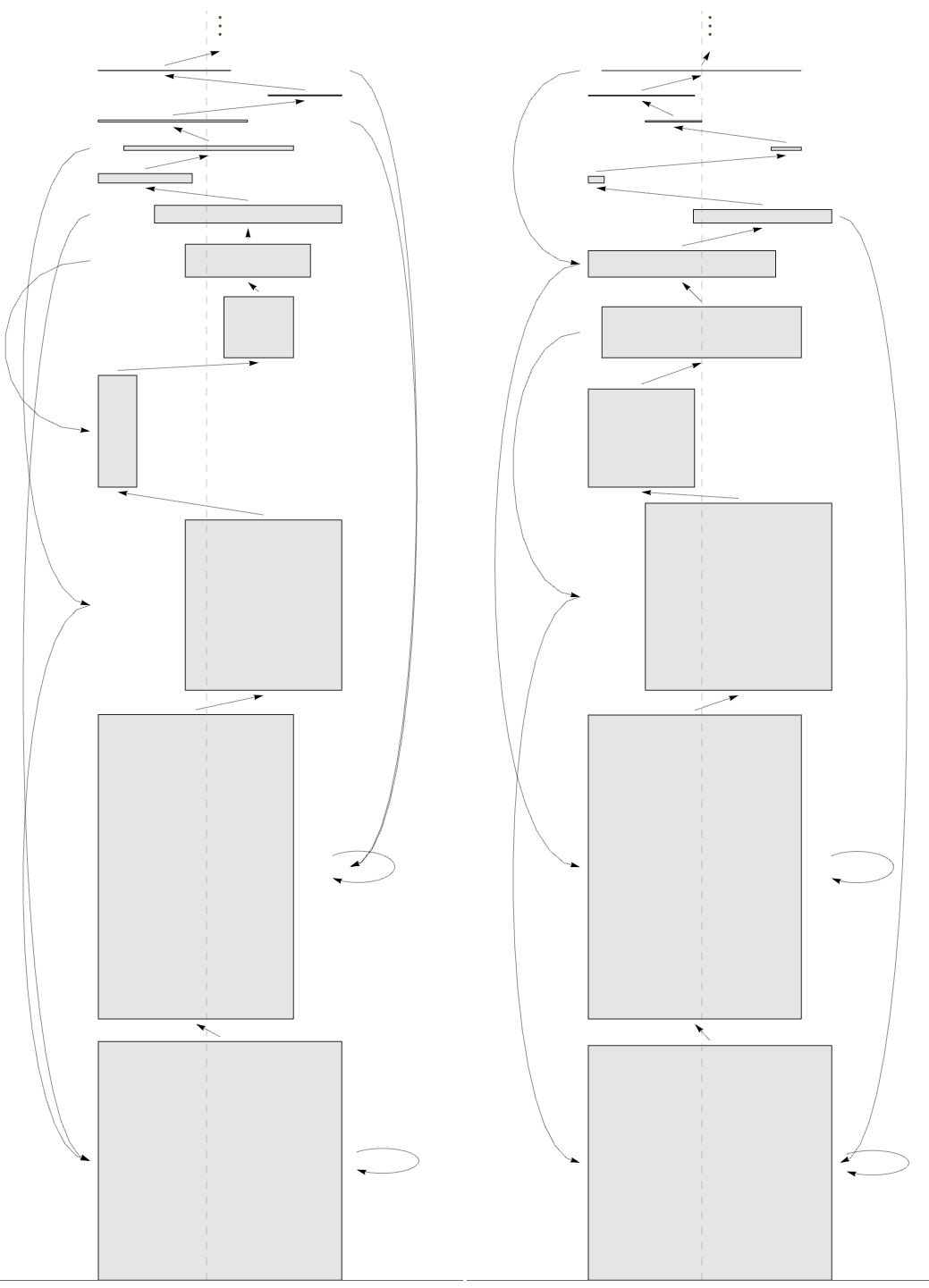} 
\caption{Natural Extension of \Cref{Example:1^k0^k} and \ref{Example:1 to 101 and 0 to 1}. (The dashed line is $x=r_{\beta}-1/\beta$.)}
\label{Fig:NE_Example1and2}
\end{figure}

For $(x_i),(y_i)\in \A^{\N}$, we define
$$
(x_i)\prec (y_i)
$$
if $(-1)^i x_i < (-1)^i y_i$ at 
the smallest index $i$ with $x_i\neq y_i$, and
$(x_i)\preceq (y_i)$ means $(x_i)=(y_i)$ or $(x_i)\prec (y_i)$.
Let $d(x,-\beta)=d_1d_2\dots$ by $d_i=\lfloor -\beta T_{-\beta}^{i-1}(x)-\ell_{\beta} \rfloor$ 
for $x\in [\ell_{\beta},r_{\beta})$. We set 
$d^*(\ell_{\beta},-\beta)=\lim_{x\searrow \ell_{\beta}} d(x, -\beta)$ and $d^*(r_{\beta},-\beta)=\lim_{x\nearrow r_{\beta}} d(x, -\beta)$. 
An infinite word $x_1x_2\dots \in \A^{\N}$ is called {\it admissible} if it is in $d([\ell_{\beta},r_{\beta}),-\beta)$.
Then, Theorem 10 in \cite{Ito-Sadahiro09} reads
the sequence $(x_i)\in \A^{\N}$ is admissible if and only if
\begin{equation}
\label{Adm}
d(\ell_{\beta},-\beta)\preceq x_nx_{n+1}\dots \prec d^*(r_{\beta},-\beta)
\end{equation}
for all $n\ge 1$.
When $\beta$ is fixed, this result
characterizes infinite words in $\A^{\N}$ which appear as a $(-\beta)$-expansion.

Let
$$
\eta:=
\lim_{\beta\to 1+0} d(\ell_{\beta},-\beta)=100111001001001110011\cdots
$$
be the fixed point of the substitution
$\phi: 1\to 100,\ 0\to 1$.
As an analogy of (\ref{SelfBeta}), W.~Steiner
characterized the words in $\A^{\N}$ that appear as the expansion of $\ell_{\beta}$, that is, the word of the form $d(\ell_{\beta},-\beta)$ for some $\beta>1$. 
\begin{Prop}[Steiner \cite{Steiner:13}]
\label{SteinerSelf}
Let $(e_i)\in \A^{\N}$. There exists $\beta>1$ that $d(\ell_{\beta}, -\beta)=e_1e_2e_3\cdots$.
if and only if it satisfies the following four conditions:
\begin{enumerate}
    \item $
e_1e_2e_3\cdots \preceq e_{n}e_{n+1}e_{n+2}\cdots $ for  $n=1,2,3,\dots$, 
\item $\eta \succ e_1e_2\cdots$,
\item $e_1e_2\cdots \not \in \{ e_1\cdots e_{k},e_1\cdots e_{k-1}(e_k-1)0\}^{\omega} \setminus \{(e_1\cdots e_k)^{\infty} \}$
for all $k\ge 1$ that $\eta\succ (e_1\cdots e_k)^{\infty}$,
\item $e_1e_2\cdots \not \in \{ e_1\cdots e_{k}0, e_1\cdots e_{k-1}(e_{k}+1)\}^{\omega}$
for all $k\ge 1$ that $\eta\succ (e_1\cdots e_{k-1}(e_k+1))^{\infty}$.
\end{enumerate}
\end{Prop}
Here, $\{x,y\}^{\omega}$ denotes the set of infinite words generated by the concatenation of $x$ and $y$ in arbitrary order.

In section 2 and later, we have to use $d^*(\ell_{\beta},-\beta)$ to construct the Markov diagram
instead of $d(\ell_{\beta},-\beta)$, because we use the corresponding symbolic dynamics on 
clopen sets. 
Fortunately, 
\begin{equation}
\label{OddPeriodcase}
d^*(\ell_{\beta},-\beta)=d(\ell_{\beta},-\beta)
\end{equation}
holds for most values of $\beta$, except when the orbit of $\ell_{\beta}$ by $T_{-\beta}$ forms an odd cycle. 
Indeed, if the orbit does not
revisit $\ell_{\beta}$, then the coding map
is continuous at $\ell_{\beta}$. Further, if it
forms an even cycle, then it is right continuous at $\ell_{\beta}$.

First, we discuss exceptional periodic expansions that cause some technical difficulty. 
We found such phenomena in \Cref{OddCycle} and \textcolor{red}{\Cref{Exception}}. See also
\cite[Proposition 5]{Ito-Sadahiro09} and its remark afterward. 

\begin{Lem}\label{OddCycle}
If the expansion is purely periodic:
$d(\ell_{\beta},-\beta)=(c_1c_2\dots c_{\ell})^{\infty}$, then $c_{\ell}\ge 1$. Further, if $\ell$ is odd, then
we have
$$
d^*(\ell_{\beta},-\beta)=(c_1c_2\dots (c_{\ell}-1)0)^{\infty}
$$
and
$$
d^*(r_{\beta},-\beta)=(0c_1c_2\dots (c_{\ell}-1))^{\infty}
$$
and the set of 
sequence $(x_i)\in \A^{\N}$ which satisfies (\ref{Adm}) but does not satisfy
$$
d^*(\ell_{\beta},-\beta)\preceq x_nx_{n+1}\dots \prec d^*(r_{\beta},-\beta)
$$
is countable. Such elements have the suffix $(c_1c_2\dots c_{\ell})^\infty$. 
\end{Lem}

\begin{proof}
If $d(\ell_{\beta},-\beta)=(c_1c_2\dots c_{\ell-1}0)^{\infty}$, then
$d(T_{-\beta}^{\ell-1}(\ell_{\beta}),-\beta)=(0c_1c_2\dots c_{\ell-1})^{\infty}$
which implies the contradiction
$T_{-\beta}^{\ell-1}(\ell_{\beta})=r_{\beta}$,
since $\ell_{\beta}=(-\beta)r_{\beta}$.
If $d(\ell_{\beta},-\beta)=(c_1c_2\dots c_{\ell})^{\infty}$,
then $T_{-\beta}^{\ell}(\ell_{\beta})=\ell_{\beta}$.
If $\ell$ is odd, then $T_{-\beta}^{\ell}$ is right discontinuous at 
discontinuity points and we have $d^*(\ell_{\beta},-\beta)=
\lim_{\eps\downarrow 0}d(\ell_{\beta}+\eps,-\beta)=
(c_1c_2\dots (c_{\ell}-1) 0)^{\infty}$ since $\ell_{\beta}=(-\beta)r_{\beta}$
implies $\lim_{\eps\downarrow 0} T_{-\beta}(r_{\beta}-\eps)=\ell_{\beta}$.
We also have $d^*(r_{\beta},-\beta)=\lim_{\eps\downarrow 0}d(r_{\beta}-\eps,-\beta)=
(0c_1c_2\dots (c_{\ell}-1))^{\infty}$.

For the last statement, let $(x_i)$ be a sequence satisfying this condition.
Then if $x_nx_{n+1}\dots x_{n+\ell-1}=
c_1\dots c_{\ell}$, then $x_{n+\ell} \ge c_1$ implies
$x_{n+\ell}=c_1$. We also have $c_2\le x_{n+\ell+1}\le c_2$ which implies $x_{n+\ell+1}=c_2$. Repeating this,
we see that $x_nx_{n+1}\dots = (c_1c_2\dots c_{\ell})^{\infty}$.
\end{proof}


\begin{Ex}\label{Example:1^k0^k}
Let $(e_i)=\sum_{k=1}^{\infty}
1^k0^k=10110011100011110000\dots
$.
Since the prefix $101$ appears only once, we can easily see that 
conditions of \Cref{SteinerSelf} are satisfied. 
There exists $\beta\approx 1.80266$ and we see that $d(\ell_{\beta},-\beta)=(e_i)_{i\ge0}$.

One can also show that $\beta$ is transcendental.
Indeed, 
\begin{align*}
 -\frac{\beta}{1+\beta}&=
\sum_{i=1}^{\infty} \frac{e_i}{(-\beta)^i}=
\sum_{k=1}^{\infty}
\frac{1}{(-\beta)^{k^2-k+1}}\frac {1-1/(-\beta)^{k}}{1-1/(-\beta)}\\
&=\frac{-1/\beta+(-1/\beta)^{3/4}\theta_{2}(-1/\beta)+
\theta_{3}(-1/\beta)/\beta}{2(1+1/\beta)}
\end{align*}
where $\theta_2,\theta_3$ are the elliptic theta null values:
$$
\theta_2(q)=\sum_{n\in \Z} q^{(n+1/2)^2}, \qquad
\theta_3(q)=\sum_{n\in \Z} q^{n^2}
$$
with $|q|<1$. By the algebraic independence result of theta null values, $\beta$ cannot be algebraic, see Theorem 4 of \cite{Bertrand} and \cite{Nesterenko}. For $d\ge2$, $\sum_{k=1}^{\infty}d^k(d-1)^k\cdots1^k0^k$ 
does not satisfy Condition 1.
\end{Ex}

\begin{Ex}\label{Example:1 to 101 and 0 to 1}
Let $\psi$ be a substitution defined by $1\to 101$ and $0\to 1$. We claim that the fixed point
$$
w:=(e_i)=\lim_{n\to \infty}\psi^n(1)=101 110110110111011011101\dots
$$
satisfies Conditions 1 and 2. We introduce
a $\Z/2\Z$-extension $\overline{\psi}$: 
$$1\mapsto 101,\quad 0\mapsto \overline{1}, \quad \overline{1}\mapsto \overline{1}\overline{0}\overline{1}, \quad \overline{0}\mapsto 1.
$$
Here, the letter $1$ is in the odd position, and $\overline{1}$ is in the even position. Similarly 
$0$ is in the even position and $\overline{0}$ is in the odd position. 
$$
\overline{w}:=
\lim_{n\to \infty}\overline{\psi}^n(1)=101\overline{1}101\overline{1}\overline{0}\overline{1}101\overline{1}101\overline{1}\overline{0}\overline{1}1\overline{1}\overline{0}\overline{1}\dots
$$
In other words, the overline indicates that the letter $\overline{1}$ (resp. $\overline{0}$) can not be replaced with $0$ (resp. $1$)
by Condition 1. Therefore
\begin{equation}
\label{forbidden}
    1010, 10111010, 101110111, 1011101100, 10111011011010,  \dots
\end{equation}
are forbidden by Condition 1. To prove the claim, we show that these words never appear as a factor of $w$. By construction, $00$, $010$ are forbidden in $w$ and
$0\overline{1}$, $\overline{1}0$ are forbidden in $\overline{w}$, in particular, the first $1010$ does not appear. If a factor in (\ref{forbidden}) appears in $w$, then it contains either one of the above forbidden words, or it
must have a preimage of $\phi$ which forms a shorter forbidden word in (\ref{forbidden}). This gives a contradiction, and we have established the claim.

Since $00$ is a forbidden word and $e_k=1$, Condition 3 of \Cref{SteinerSelf} is fulfilled. Similarly, Condition 4
is valid, since $e_k=0$. Therefore, Proposition \ref{SteinerSelf} guarantees that we have
$w=d(\ell_{\beta},-\beta)$
with a unique $\beta \approx 1.875305172$.

\end{Ex}

\section{Hofbauer's Markov Diagram}\label{Hofbauer tower and Markov Diagram}

In this section, we recall piecewise monotone maps on the unit interval $[0,1]$ (of course, it can also be $[\ell_{\beta}, r_{\beta}]$) and their Markov diagrams.
\begin{Def}[piecewise monotone map]
We call $T:[0,1]\to[0,1]$ a piecewise monotone map if there exist disjoint intervals $I_0,\cdots, I_k\subset [0,1]$
with the partition points $c_0=0<c_1< \ldots < c_{k+1}=1$
satisfying the following conditions:
\begin{enumerate}
    \item $[0,1]=\displaystyle \bigcup_{j=0}^k I_j$,
    \item $T|_{I_j}$ is continuous and strictly monotone,
    \item $\displaystyle \bigcup_{n=1}^{\infty}T^{-n}\left([0,1]\setminus \bigcup_{j=0}^k Int(I_j)\right)$ is dense in $[0,1]$.
\end{enumerate}
\end{Def}

Let
\begin{equation*}\label{clopenset}
X_{T}:=\bigcap_{n=1}^{\infty}T^{-n}\left(\bigcup_{j=0}^k Int(I_j)\right),
\end{equation*}
and define the coding map 
$\Psi:X_T\to \Sigma_k^{+}=\{0,1,2,\dots,k\}^{\mathbb{N}}$ by $\Psi(x)=j_1j_2\cdots$, where $j_i$ is the index 
with $T^{i-1}(x)\in I_{j_i}$. Set 
$$
\Sigma_T^{+}=\{{\bf x}=x_1x_2\cdots\in \Sigma_k^{+}\mid a^{x_m}\preceq \sigma^{m-1}({\bf x})=x_{m}x_{m+1}\cdots\preceq b^{x_m}
\mbox{~ for any } m \in \mathbb{N} \},
$$ 
where $\sigma$ is a shift transformation and 
\begin{align*}
a^{j}=\lim_{x\in I_j, ~x\searrow c_{j}}\Psi(x), \quad
 b^{j}=\lim_{x\in I_{j}, ~x\nearrow c_{j+1}}\Psi(x).
\end{align*}

Then, $\Psi$ is injective, and the following diagram is commutative. It is surjective except for a countable set \cite[Lemma 2]{Hofbauer79}:
\[
  \begin{CD}
     {X_T} @>{T}>> {X_T} \\
  @V{\Psi}VV    @V{\Psi}VV \\
     {\Sigma_{T}^{+}} @>{\sigma}>> {\Sigma_{T}^{+}}
  \end{CD}
\]

Here, we denote 
$\Psi(I_{w_1} \cap T^{-1}(I_{w_2}) \cap \cdots\cap T^{-n+1}(I_{w_n}) \cap X_T)$ by $[w_1w_2\cdots w_n]$
for $n \in \mathbb{N}$ and 
$w_1w_w \cdots w_n \in \{0,1,\ldots,k \}^\mathbb{N}$.

We now define the Markov diagram of Hofbauer, which is a countable directed graph whose vertex set is   $\Psi(X_{T})$.

\begin{Def}[Hofbauer's Markov Diagram]\label{Def:Markov Diagram}
    Let $X\subset \{0,1,2,\dots, k\}^{\mathbb{N}}$ and $C, D\subset X$. We denote $C \rightarrow D$ if there exists $i\in\{0,1,2,\dots, k\}$ such that $D=\sigma(C)\cap [i]\not=\emptyset$.
Define $\mathcal{D}_i$ inductively as follows:
\begin{align*}
    \mathcal{D}_0:&=\{[0],[1],[2],\dots, [k]\}, \\
    \mathcal{D}_n:&=\{D \mid \exists C\in \mathcal{D}_{n-1}~ s.t.~ C\rightarrow D \} \quad \text{for~} n\in\mathbb{N}.
\end{align*}
Set $\mathcal{D}_X=\bigcup_{n=0}^{\infty}\mathcal{D}_n$. The pair $(\mathcal{D}_X,\rightarrow)$ is called Hofbauer's Markov diagram. 
\end{Def}

When $C \to D$, there exists unique $i \in \{0,1,\ldots, k\}$ such that $D \subset [i]$, and we label the arrow from $C$ to $D$ as $i$ and 
write  $C \xrightarrow{i} D$.
We consider the set of infinite paths on Hofbauer's Markov diagram $(\mathcal{D}_X, \rightarrow)$:
\[
\Sigma_{\mathcal{D}_X}^+ := \{ (D_i)_{i \in \mathbb{N}}
\in \mathcal{D}_X^{\mathbb{N}} \mid D_{i} \rightarrow D_{i+1} \mbox{ for any } n \in \mathbb{N}\},
\]
and define $\Phi^+: \Sigma_{\mathcal{D}_X}^+  \to \{1,2,\ldots,k\}^\mathbb{N}$ by
$\Phi((D_i)_{i \in \mathbb{N}}) = (e_i)_{i \in \mathbb{N}}$ for $(D_i)_{i \in \mathbb{N}} \in \mathcal{D}_X^{\mathbb{N}} $
with $D_i \xrightarrow{e_i} D_{i+1}$. Then
$\Sigma_T^+ = \Phi^+( \Sigma_{\mathcal{D}_X}^+)$ (See \cite{Hofbauer86}).


For the original definition and study on Hofbauer's Markov diagram, we can see \cite{Hofbauer79}, \cite{Hofbauer81II}, and \cite{Hofbauer81}. In particular, \cite{Hofbauer86} includes the case of the decreasing map $T$, and \cite{BruinKalle} provides several examples of towers for $\beta$-transformation and $(-\beta)$-transformation.

\textcolor{blue}{
Before starting the case of $(-\beta)$-transformation $T_{-\beta}$, let us recall
the case of $\beta$-transformation $T_\beta$. It is known that the Markov diagram is obtained in the following way by using the expansion 
 $d^*(1,\beta)= b_1 b_2 \cdots$ of 1.
 \textcolor{red}{
 The cylinders $[0],\dots,[b_1-1]$ work in the same way and merge into a single state $0$. The cylinder $[b_1\dots b_{n-1}]$ corresponds to the state $n$ for $n\ge 1$.
 }
The vertex set of the Markov diagram is $\mathbb N \cup \{0\}$, and there exists an uparrow from the state $n$ to the state $n+1$ labeled by $b_{n+1}$ and
downarrows from the state $n$ to the state $0$ labeled by $0,1,\ldots,b_{n+1}-1$ for $n \in \mathbb N \cup \{0\}$
(See \Cref{Markov diagram for beta-transformation}).}

\begin{figure}[htbp]
    \centering
\begin{tikzpicture}[auto=right, bezier bounding box,
     node distance =15mm,
     start chain = going right, 
 every edge/.style = {draw, -Stealth, semithick},
 every state/.style = {draw, thick, on chain}]
    \begin{scope}[nodes=state]
\node (A) {$0$};
\node (B) {$1$};
\node (C) {$2$};
\node (D) {$n$};
\coordinate (E);
    \end{scope}
\path   (A) edge[loop left, looseness=12,  "$0\text{,}...\text{,~}b_1-1$" '] (A);
\path[left, swap] 
        (A) edge["$b_1$"]  (B)
        (B) edge["$b_2$"]  (C)
        (C) edge[dashed]  (D)
        (D) edge["$b_{n+1}$"]  (E);
\path[out=220, swap]
        (B) edge[in=-30, "$0\text{,}... \text{,~}b_2-1$"]  (A)
        (D) edge[in=-45, "$0\text{,}... \text{,~}b_{n+1}-1$"]  (A);
\path[out=135]
        (C) edge[in=40, "$0\text{,}... \text{,~}b_{3}-1$"]  (A);
    \end{tikzpicture}
\caption{Markov diagram for the $\beta$-transformation}
    \label{Markov diagram for beta-transformation}
\end{figure}
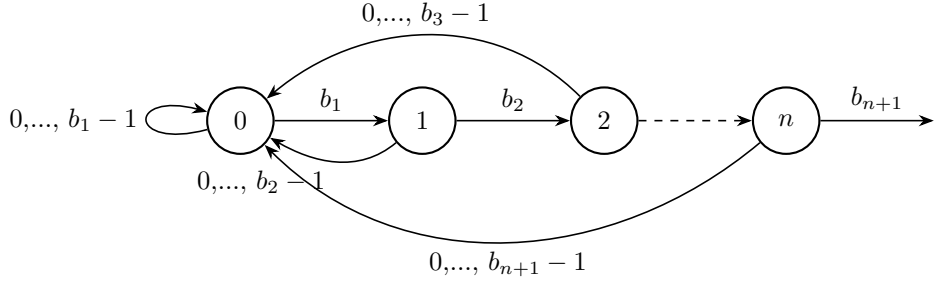

From now on, we study the case of ($-\beta$)-transformation.
Let $d^{*}(\ell_{\beta}, -\beta)=b_1b_2\cdots$.
Further, let the partition points $c_0=r_\beta >c_1> \ldots >c_{b_1+1}=\ell_\beta$ be
$c_i=r_\beta-i/\beta$ for $i \in \{0,1,\ldots,b_1\}$ and 
$I_0 =(c_1,c_0], I_1=(c_2,c_1], I_2=(c_3,c_2],\ldots,
I_{b_1}=[c_{b_1+1}, c_{b_1}]$.

The ($-\beta)$-transformation $T_{-\beta}$ is a piecewise monotone map, and we can do the same as in the case of the map on $[0,1]$.

For brevity, we write 
$T_{-\beta}^{n-1} (\ell_{\beta}^+) \in I_{b_n}$,
to mean that
$T_{-\beta}^{n-1}(\ell_{\beta}+\varepsilon)\in
I_{b_n}$ 
for all sufficiently small $\varepsilon>0$.
\textcolor{red}{
By the definition of $T_{-\beta}$, we have $b_2<b_1$.}
Since $\overline{\Psi^{-1}(\sigma([i]))}=[\ell_{\beta},r_{\beta}]=\cup_{j=0}^{b_1} \overline{I_j}$ for 
$i\not=b_1$, 
$[i]\xrightarrow{j} [j]$ for $i \in\{0,1,\dots,\textcolor{blue}{b_1-1}\}$ and $j \in\{0,1,\dots,b_1\}$.
Since 
$T_{-\beta}(\ell_{\beta}^+) \in I_{b_2}$
and
$\overline{\Psi^{-1}(\sigma([b_1]))}=[\ell_{\beta},T_{-\beta}(\ell_{\beta}^+)]=\cup_{j=b_2+1}^{b_1}\overline{I_j} 
\cup [c_{b_2+1},T_{-\beta}(\ell_{\beta}^+)]$, 
$$
\begin{array}{l}
[b_1]\xrightarrow{j} [j] \quad \quad \mbox{ for } j\in\{b_2+1,\dots,b_1-1\},\\
[10pt]
[b_1] \xrightarrow{b_1}[b_1],\\
[10pt]
[b_1]\xrightarrow{b_2} \sigma([b_1])\cap[b_2]=\sigma([b_1b_2])=\Psi([c_{b_2+1}, T_{-\beta}(\ell_{\beta}^{+}))\cap X_T).
\end{array}
$$
Similarly, since
$T_{-\beta}^2(\ell_{\beta}^+) \in I_{b_3}$
and
$\overline{\Psi^{-1}(\sigma^2([b_1b_2]))}=[T_{-\beta}^2(\ell_\beta^+),r_\beta]=\cup_{j=0}^{b_3-1}\overline{I_j} \cup [T_{-\beta}^2(\ell_\beta^+),c_{b_3}]$,
$$
\begin{array}{l}
\sigma[b_1b_2] \xrightarrow{j} [j] \quad \quad \mbox{ for } j \in\{0,1,\dots,b_3-1\},\\
[10pt]
\sigma[b_1b_2] \xrightarrow{b_3} \sigma^2[b_1b_2b_3],
\end{array}
$$
and we can see
\begin{align*}
    &\mathcal{D}_{1}\setminus \mathcal{D}_0=\{\sigma([b_1b_2])\},~ 
    \mathcal{D}_{2}\setminus (\mathcal{D}_0\cup \mathcal{D}_1)=\{\sigma^2([b_1b_2b_3])\}, \cdots.
\end{align*}
In the case of a $(-\beta)$-transformation, there is only one non-full part such that  
\[
\mathcal{D}_{n}\setminus \bigcup_{i=0}^{n-1}\mathcal{D}_i=\{\sigma^n([b_1b_2\cdots b_{n+1}])\}.  
\]
\Cref{Markov diagram} gives a figure of Hofbauer's Markov diagram.
In fact, by letting the state $0$ of the diagram given by \Cref{Markov diagram} denote the collection of $[0], [1], \ldots, [b_1-1]$ and the state $n$ for $n\in \mathbb{N}$ denote $\sigma^{n-1}([b_1 b_2 \cdots b_{n}])$, this diagram becomes Hofbauer's Markov diagram.

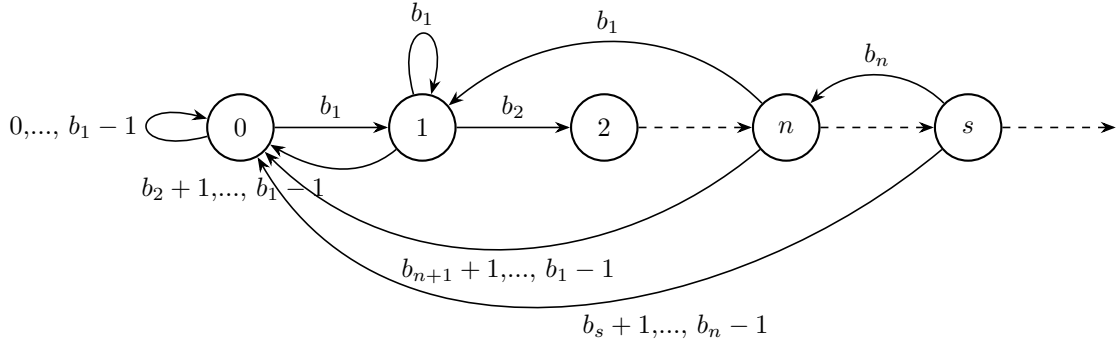
\begin{figure}[htbp]
    \centering
\begin{tikzpicture}[auto=right, bezier bounding box,
     node distance =15mm,
     start chain = going right, 
 every edge/.style = {draw, -Stealth, semithick},
 every state/.style = {draw, thick, on chain}]
    \begin{scope}[nodes=state]
\node (A) {$0$};
\node (B) {$1$};
\node (C) {$2$};
\node (D) {$n$};
\node (E) {$s$};
\coordinate (F);
    \end{scope}
\path   (A) edge[loop left, looseness=12,  "$0\text{,}...\text{,~}b_1-1$" '] (A)
        (B) edge[loop above, looseness=12,  "$b_1$" '] (B);
\path[left, swap] 
        (A) edge["$b_1$"]  (B)
        (B) edge["$b_2$"]  (C)
        (C) edge[dashed]  (D)
        (D) edge[dashed]  (E)
        (E) edge[dashed]  (F);
\path[out=220, swap]
        (B) edge[in=-30, "$b_2+1\text{,}... \text{,~}b_1-1$"]  (A)
        (D) edge[in=-45, "$b_{n+1}+1\text{,}... \text{,~}b_{1}-1$"]  (A)
        (E) edge[in=-60, "$b_s+1\text{,}... \text{,~}b_{n}-1$"]  (A);
\path[out=135]
        (D) edge[in=40, "$b_1$"]  (B)
        (E) edge[in=45, "$b_{n}$"]  (D);
    \end{tikzpicture}
\caption{Hofbauer's Markov diagram when $n, \textcolor{red}{s}$ is odd number}
    \label{Markov diagram}
\end{figure}

\textcolor{blue}{In the above discussion, the interval
$\overline{\Psi^{-1}(\sigma^n([b_1b_2\cdots b_n]))}$ 
and its endpoint $T_{-\beta}^m(\ell_\beta^+)$ greatly help in finding edges from the state $n$ in Hofbauer's Markov diagram. We wish to continue
this discussion using a realization as a tower. 
Let $R_0=[\ell_\beta,r_\beta]$  and
$R_n=\overline{\Psi^{-1}(\sigma^n([b_1b_2\cdots b_n]))}$
for $n \in \mathbb{N}$,
and stack these intervals vertically to geometrically realize  Hofbauer's Markov diagram as in \Cref{Hofbauer tower}. 
Later, we will see in Theorem \ref{Th:left and right eigenvector} that the exact locations and lengths of these intervals
correspond to the shape of the right eigenvector of the incidence matrix of the Markov diagram.
In what follows, we construct Hofbauer's Markov diagram by using its tower realization depicted in \Cref{Hofbauer tower}.}

\begin{figure}[ht]
    \centering
    \begin{tikzpicture}[scale=0.96][domain=-0.6:6]
\draw[->] (2.7,10.2) -- (2.7,10.8);
\fill (2.7,10.5) node[right] {$b_{s+1}$};
\draw (1.5,10) node[left] {$T_{-\beta}^{n-1}(\ell_{\beta}^+)$} -- (2.8,10) node[right] {$T_{-\beta}^{s}(\ell_{\beta}^+)$};
\fill (2,9.5) circle (0.5pt);
\fill (2,9) circle (0.5pt);
\fill (2,8.5) circle (0.5pt);
\draw (2,8) node[left] {$T_{-\beta}^{m+1}(\ell_{\beta}^+)$} -- (2.3,8) node[right] {$T_{-\beta}(\ell_{\beta}^+)$};
\draw[->] (0.15,7.2) -- (2.15,7.8);
\draw (0,7) node[left] {$\ell_{\beta}$} -- (0.3,7) node[right] {$T_{-\beta}^{m}(\ell_{\beta}^+)$};
\fill (2,6.5) circle (0.5pt);
\fill (2,6) circle (0.5pt);
\fill (2,5.5) circle (0.5pt);
\draw (1.8,5) node[left] {$T_{-\beta}^{n+1}(\ell_{\beta}^+)$} -- (4,5) node[right] {$r_{\beta}$};
\draw[->] (2.8,4.2) -- (2.8,4.8);
\fill (2.8,4.5) node[right] {$b_{n+1}$};
\draw (0,4) node[left] {$\ell_{\beta}$} -- (3.5,4) node[right] {$T_{-\beta}^{n}(\ell_{\beta}^+)$};
\fill (2,3.5) circle (0.5pt);
\fill (2,3) circle (0.5pt);
\fill (2,2.5) circle (0.5pt);
\draw (1.5,2) node[left] {$T_{-\beta}^2(\ell_{\beta}^+)$} -- (4,2) node[right] {$r_{\beta}$} ;
\draw[->] (2.7,1.2) -- (2.7,1.8);
\fill (2.7,1.5) node[right] {$b_2$};
\draw (0,1) node[left] {$\ell_{\beta}$} -- (3,1) node[right] {$T_{-\beta}(\ell_{\beta}^+)$} ;
\draw[->] (0.5,0.2) -- (0.5,0.8);
\fill (0.5,0.5) node[right] {$b_1$};
\draw (0,0) node[left] {$\ell_{\beta}$} -- (4,0) node[right] {$r_{\beta}$};

\draw (1.5,10) circle (1.5pt);  
\draw (2.8,10) circle (1.5pt);  

\draw (2,8) circle (1.5pt);  
\draw (2.3,8) circle (1.5pt);  

\draw (0,7) circle (1.5pt); 
\draw (0.3,7) circle (1.5pt); 

\draw (1.8,5) circle (1.5pt);  
\draw (4,5) circle (1.5pt);  

\draw (0,4) circle (1.5pt);  
\draw (3.5,4) circle (1.5pt);  

\draw (4,2) circle (1.5pt);
\draw (1.5,2) circle (1.5pt);

\draw (0,1) circle (1.5pt);  
\draw (3,1) circle (1.5pt);  

\draw (0,0) circle (1.5pt);
\draw (4,0) circle (1.5pt);   

\draw[->] (-0.5,10) -- (-1,10) -- (-1,4.1) --  (-0.5,4.1);
\draw[->] (5,10) -- (6,10) -- (6,0) --  (5,0);
\draw[->] (-0.5,3.9) -- (-1,3.9) -- (-1,1) -- (-0.5, 1);
\draw[->] (-0.5,4) -- (-1.5,4) -- (-1.5,0) -- (-0.5, 0);
\draw[->] (5,0.2) -- (6.5,0.2) -- (6.5,-0.2) -- (5, -0.2);
\draw[->] (5,1.2) -- (6.5,1.2) -- (6.5,0.8) -- (5, 0.8);
\draw[->] (-0.5,0.8) -- (-2,0.8) -- (-2,-0.2) -- (-0.5, -0.2);

\fill (-1,7) node[left] {$d=b_{n}$};
\fill (6,5) node[right] {$b_{s+1}+1, \dots, b_n-1$};
\fill (2,7.5) node[right] {$d_{m+1}=b_1$};
\fill (-1,2.5) node[left] {$b_{1}$};
\fill (-1.5,2) node[left] {$b_{n+1}+1, \dots , b_{1}-1$};
\fill (6.5,0) node[right] {$0, \dots , b_1-1$};
\fill (6.5,1) node[right] {$b_1$};
\fill (-2,0.3) node[left] {$b_2+1, \dots , b_1-1$};

\draw plot[mark=x, mark size=2pt] coordinates {(1,4)}; 
\draw plot[mark=x, mark size=2pt] coordinates {(1,1)};
\draw plot[mark=x, mark size=2pt] coordinates {(1,0)};

\draw plot[mark=x, mark size=2pt] coordinates {(2.5,10)};
\draw plot[mark=x, mark size=2pt] coordinates {(2.5,5)};
\draw plot[mark=x, mark size=2pt] coordinates {(2.5,4)}; 
\draw plot[mark=x, mark size=2pt] coordinates {(2.5,2)}; 
\draw plot[mark=x, mark size=2pt] coordinates {(2.5,1)};
\draw plot[mark=x, mark size=2pt] coordinates {(2.5,0)};
\end{tikzpicture}
    \caption{Realization of Hofbauer's Markov Diagram when $n,m$ are odd number}
    \label{Hofbauer tower}
\end{figure}
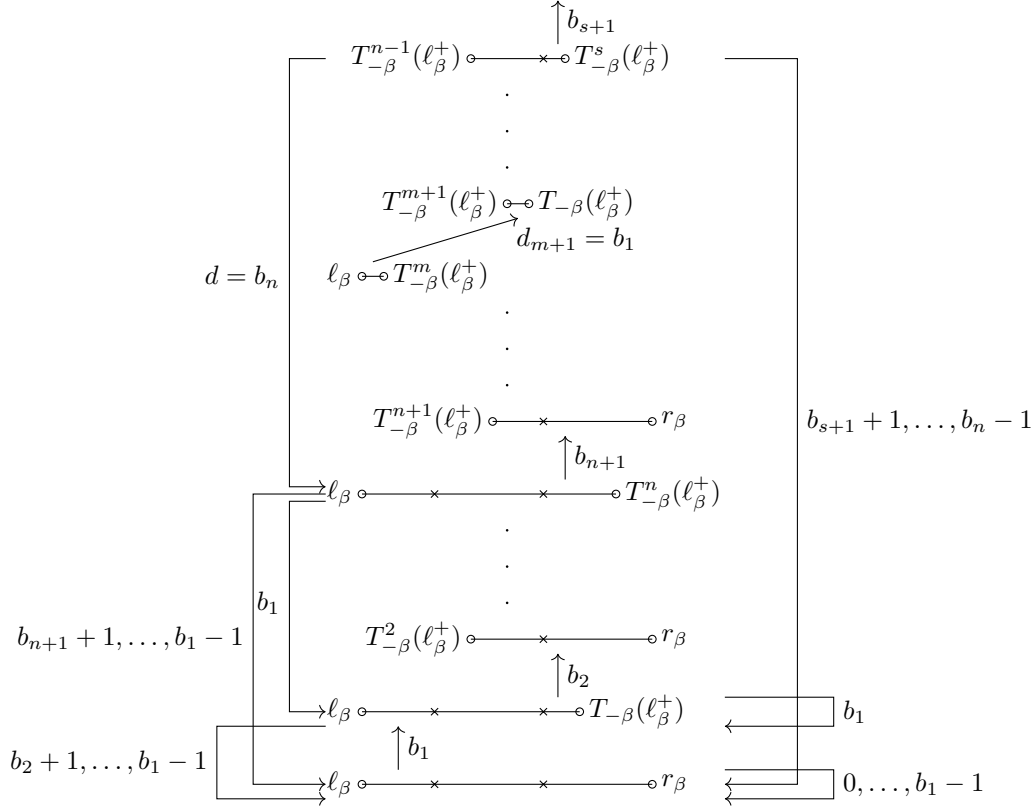

\textcolor{blue}{The Markov diagram is constructed level by level from the bottom for $m=0,1,2,\ldots$. In this process, there are two cases, namely Case I and Case II, as described below. For each case, we examine how the edges from the $m$-th floor are drawn.}

Case I: $m$ is even, or $m$ is odd and $\textcolor{red}{T_{-\beta}^{m}(\ell_{\beta}^+) \notin I_{b_1}}$.\\
In this case, we have
\begin{align}\label{R_m}
    R_{m}=
\begin{split}
\left\{
\begin{array}{ll}
\displaystyle [\ell_\beta,T^{m}(\ell_{\beta}^{+})]
=\bigcup_{j=b_{m+1}+1}^{b_1}\overline{I_j} 
\cup [c_{b_{m+1}+1},T^m_{-\beta}(\ell_{\beta}^+)], 
\quad & m: odd, \\
[10pt]
\displaystyle [T^{m}_{-\beta}(\ell_{\beta}^{+}),r_\beta]
=[T_{-\beta}^m(\ell_\beta^+),c_{b_{m+1}}]
 \cup \bigcup_{j=0}^{b_{m+1}-1}\overline{I_j}, 
\quad & m: even.
\end{array} \right. 
\end{split}
\end{align}
Therefore, 
from the $m$-th floor 
there exists one uparrow to the $(m+1)$-th floor
labeled by $b_{m+1}$.
And there exist downarrows to the $0$-th floor
labeled by $0,1,\ldots,b_{m+1}-1$ 
when $m$ is even, or
downarrows to the $0$-th floor
labeled by $b_{m+1}+1, b_{m+1}+2,\ldots, b_1 -1$  and
one downarrow to the 1st floor
labeled by $b_1$ 
when $m$ is odd and
$T_{-\beta}^{m}(\ell_{\beta}^+) \notin I_{b_1}$.

Case II: $m$ is odd and $T_{-\beta}^{m}(\ell_{\beta}^+) \in I_{b_1}$.\\
In this case, 
\begin{align}\label{R_m+a}
    R_{m+a}=
    \begin{split}
\left\{
\begin{array}{ll}
\displaystyle [T^{m+a}_{-\beta}(\ell_{\beta}^{+}),T^{a}(\ell_{\beta}^{+})], \quad a:odd, \\
[10pt]
\displaystyle [T^{a}(\ell_{\beta}^{+}), T^{m+a}_{-\beta}(\ell_{\beta}^{+})], \quad a:even,
\end{array} \right. 
\end{split}
\end{align}
and 
$R_{m+a} \subset \overline{I_{b_a}}$
\textcolor{blue}{for some $a=0,1,\cdots$.
Since $T_{-\beta}$ is expanding, there exists the first index $a\ge 1$ that 
$R_{m+a}$ contains some of the partition points $\{c_1, c_2, \ldots,c_{b_1} \}$.
We denote this index $a$ by $n-1$.} Then $b_{m+1}=b_1, b_{m+2}=b_2, \ldots,b_{m+n-1}=b_{n-1}$ and
$b_{m+n} \ne b_{n}$ and

\begin{align}\label{R_m+n}
R_{m+(n-1)}=
    \begin{split}
\left\{
\begin{array}{ll}
\displaystyle 
[T^{n-1}_{-\beta}(\ell_{\beta}^+),c_{b_{n}}] 
\cup \bigcup_{j=b_{m+n}+1}^{b_{n}-1} \overline{I_j} 
\cup [c_{b_{m+n}+1},T^{m+n-1}_{-\beta}(\ell_{\beta}^+)], \quad & n: odd, \\
[13pt]
\displaystyle
 [T^{m+n-1}_{-\beta}(\ell_{\beta}^+),c_{b_{m+n}}]
\cup \bigcup_{j=b_{n}+1}^{b_{m+n}-1}\overline{I_j} 
\cup [c_{b_{n}+1},T^{n-1}_{-\beta}(\ell_{\beta}^+)], \quad & n: even.
\end{array} \right. 
\end{split}
\end{align}

Therefore, for $a \in \{0,1,\ldots, n-2\}$, there is no arrow from the $(m+a)$-th floor to any lower floor, and
there exists only one uparrow to the $(m+a+1)$-th floor labeled by $b_{m+a+1}$. 
From $(m+n-1)$-th floor, there exists one uparrow to the $(m+n)$-th floor labeled by $b_{m+n}$, and there exists downarrows to the $0$-th floor labeled by $b_{m+n}+1, b_{m+n}+2,  \ldots, b_{n}-1$
when $n$ is odd,
or
downarrows to the $0$-th floor labeled by $b_{n}+1, b_{n}+2, \ldots, b_{m+n}-1$, when $n$ is even. 
Note that $R_m \subset \overline{I_{b_1}}$ and
$R_{m+a} \subset R_a$ for $a \in \{1,2,\ldots,n-\textcolor{blue}{1}\}$. 
Since $[T^{n-1}_{-\beta}(\ell_{\beta}^+), c_{b_{n}}]\subset 
\textcolor{blue}{R_{n-1}}$ 
when $n$ is odd or $[c_{b_{n}+1}, T^{n-1}_{-\beta}(\ell_{\beta}^+)] \subset 
\textcolor{blue}{R_{n-1}}$ 
when $n$ is even, there is one downarrow to the $n$-th floor labeled by $b_{n}$ (In \Cref{Markov diagram} and \ref{Hofbauer tower}, we denote $s:=m+n-1$).
\textcolor{blue}{
Moreover, we can see $R_{n}=[\ell_\beta, T^{n}(\ell_\beta^+)]$ when $n$ is odd or $R_{n}=[T^{n}(\ell_\beta^+),r_\beta]$ when $n$ is even. 
For $a=1,2,\ldots,n-1$, 
$R_n \ne R_{m+a}$, thus $n \le m$.}

Summarizing these discussions, Hofbauer's Markov Diagram is given by the following rule:

\begin{Th}[How to construct the Markov diagram]\label{How to construct the Markov diagram}
Put $d^*(\ell_{\beta},-\beta)=b_1b_2\dots .$
Let the collection of the cylinders $[0],[1],\ldots,[b_1-1]$ be the state $0$ and
$\sigma^{n-1}([b_1b_2\cdots b_n])$ the state $n$ for
$n \in \mathbb N$.
The edges of Hofbauer's Markov diagram are given by the following way:\\
\textcolor{blue}{There are self-loops and uparrows in the diagram as shown below.}
$$
\begin{array}{l}
0 \xrightarrow{i} 0 \quad \quad \quad \mbox{ for } i \in\{0,1,\dots,b_1-1\},\\
[10pt]
1 \xrightarrow{b_1} 1,\\
[10pt]
m \xrightarrow{b_{m+1}} m+1 \quad \quad \mbox{ for } m \in \mathbb N \cup \{0\}.
\end{array}
$$

We draw downarrows by the following infinite algorithm 
starting from the state $m\mapsto 1$.
\medskip

\textcolor{red}{
$(*)$ Unless $b_{m+1}=b_1$ with some odd number $m$, draw downarrows as follows:} 
$$
\begin{array}{ll}
m \xrightarrow{i} 0 \quad \quad \mbox{ for } i \in
\{0,1,\dots,b_{m+1}-1\}  & m : even,\\
[10pt]
\begin{aligned}
&m \xrightarrow{i} 0  \quad \quad \mbox{ for } i \in 
\{ b_{m+1}+1, b_{m+1}+2, \ldots,b_1-1\},\\
&m \xrightarrow{b_1} 1
\end{aligned} 
\qquad 
& m: odd.
\end{array}
$$
\textcolor{red}{Then we move to the next state $m\mapsto m+1$ and return to $(*)$.}
\medskip

\textcolor{blue}{
In the case where 
$b_{m+1}=b_1, b_{m+2}=b_2, \ldots,b_{m+n-1}=b_{n-1}$ and
$b_{m+n} \ne b_{n} $ with some odd number $m$ and 
$n \le m$, 
}
Draw downarrows from the state $s:=m+n-1$ as follows:
$$
s \xrightarrow{b_{n}} n,
$$
and
$$
\begin{array}{ll}
s \xrightarrow{i} 0 \quad \quad \mbox{ for } i \in
\{b_{n}+1, b_{n}+2, \ldots, b_{s+1}-1\}  & n,s : even,\\
[10pt]
s \xrightarrow{i} 0  \quad \quad \mbox{ for } i \in 
\{b_{s+1}+1, b_{s+1}+2, \ldots, b_{n}-1\}
& n,s: odd.
\end{array}
$$
\textcolor{red}{Then we move to the next state $m\mapsto s+1$ and return to $(*)$.}
\end{Th}

\textcolor{red}{
\begin{Rem}
\label{Exception}
In the last part of the above proof of \Cref{How to construct the Markov diagram}, we 
could directly see that the case where $m$ is odd and $n>m$ cannot occur, since the corresponding 
interval would be away from $\ell_{\beta}$
and $r_{\beta}$ and leads to a contradiction, if such a case exists. 
This can be shown in the following way as well.
Since $b_{m+a}=b_a$ for $a=1,2,\dots,m$,
by applying the same discussion as the last part of the proof of Lemma \ref{OddCycle}, the expansion is purely periodic of
period $(b_1\dots b_m)^{\infty}$.
But $d^*(\ell_{\beta},-\beta)$ can not be in this form
by Lemma \ref{OddCycle}.
See \Cref{Example:x^3-3x^2+x-1} for the construction of
the exceptional cases for (\ref{OddPeriodcase}).
\end{Rem}
}

We denote the countable incidence matrix of the Markov diagram by $A_{-\beta}$.
\textcolor{red}{
In the Markov diagram in \Cref{How to construct the Markov diagram}, 
we call $i-j+1$ the {\bf defect} of a downarrow $i\rightarrow j$ with $j\le i$.}

\begin{Cor}\label{Defect}
The defects of the arrows $i\rightarrow j$ with $0<j\le i$ are distinct positive integers.
\end{Cor}

\begin{proof}
By \Cref{How to construct the Markov diagram}, such \textcolor{red}{a self-loop and} 
downarrows 
have a form $s \xrightarrow{b_{n}} n$ and $s-n+1=m$,
where $m$ is the state of the Markov diagram that we visit by the above algorithm.
\end{proof}

This property will be used in the proof of \Cref{SystemPositiveRecurrent}, 
which is the clue to the construction of our natural extension.




\section{Perron--Frobenius Theorem for Countable Markov Shifts}\label{Generalized Perron-Frobenius Theorem}
 
To prove the main theorem, we recall the Perron--Frobenius theorem described in
\cite[Chapter 7]{Kitchens98} associated with the countable infinite matrix $\Ab$ with non-negative entries.

We say such a matrix is {\it irreducible} if for every pair of indices $i$ and $j$, there is $n$ with $(\Ab^n)_{ij}>0$. Equivalently, the matrix is irreducible when its incidence graph is strongly connected.

Fix an index $i$ and let $p(i)=\gcd\{n\ge 1 \mid (\Ab^n)_{ii} >0\}$. This is the period of the index $i$. When $\Ab$ is {\it irreducible} and the period of every index is the same, that is called the period of $\Ab$.
As before, a matrix with period one is said to be {\it aperiodic}. Note that $(\Ab)_{00}>0$ implies $\Ab$ is {\it aperiodic}.

We will define three generating functions. For any pair of indices $i$ and $j$, we define
\begin{align*}
a_{ij}(0)=\delta_{ij}, \quad a_{ij}(1)=\Ab_{ij}, \quad a_{ij}(n)=(\Ab^n)_{ij}.
\end{align*}
The first generating function is defined by
\begin{align*}
H_{ij}(z)=\sum_{n=0}^{\infty}a_{ij}(n)z^n.
\end{align*}
For the next type of generating functions, we define coefficients inductively.
Let
\begin{align*}
\ell_{ij}(0)=0, \quad \ell_{ij}(1)=a_{ij}(1), \quad \ell_{ij}(n+1)=\sum_{r\not=i}\ell_{ir}(n)a_{rj}(1).
\end{align*}
The coefficient $\ell_{ij}(n)$ is the sum of the weights of the paths that go from $i$ to $j$ in $n$ steps without returning to $i$ at any time prior to $n$. Define the functions
\begin{align*}
L_{ij}(z)=\sum_{n=1}^{\infty}\ell_{ij}(n)z^n.
\end{align*}

Similarly, define coefficients
\begin{align*}
r_{ij}(0)=0, \quad r_{ij}(1)=a_{ij}(1), \quad r_{ij}(n+1)=\sum_{r\not=j}a_{ir}(1)r_{rj}(n).
\end{align*}
The coefficient $r_{ij}(n)$ represents the sum of the weights of the paths that go from $i$ to $j$ in $n$ steps without hitting $j$ at any time prior to $n$. Define the functions
\begin{align*}
R_{ij}(z)=\sum_{n=1}^{\infty}r_{ij}(n)z^n.
\end{align*}

We define $\lambda$ to be the Perron value of $\Ab$. The irreducibility of $\Ab$ implies $\displaystyle \lambda=\lim_{n\rightarrow \infty}\sqrt[n]{(\Ab^n)_{ij}}$ (c.f. \cite[Lemma 7.1.1]{Kitchens98}).
The matrix $\Ab$ is {\it recurrent} if $H_{ii}(1/\lambda)=\infty$. The recurrent matrix $\Ab$ is said to be {\it positive recurrent} if 
\begin{align*}
\sum_{n=1}^{\infty}n\frac{\ell_{ii}(n)}{\lambda^n}<\infty, 
\end{align*}
and it is said to be {\it null recurrent} if $\sum_{n=1}^{\infty}n\ell_{ii}(n)/\lambda^n=\infty$.

\begin{Prop}[Generalized Perron-Frobenius Theorem, {\cite[Theorem 7.1.3]{Kitchens98}}]\label{GPFT}
Suppose $\Ab$ is a countable nonnegative matrix. Further, suppose it is irreducible, aperiodic, and recurrent. Then there exists a finite Perron value $\lambda>0$ such that:\\
[5pt]
$(a)$ $\displaystyle \lambda=\lim_{n\rightarrow \infty}\sqrt[n]{(\Ab^n)_{ij}}$ for any pair of indices $i$ and $j$, so that $1/\lambda$ is the radius
of convergence of the power series $\displaystyle H_{ij}(z)=\sum_{n=0}^{\infty}a_{ij}(n)z^n$,\\
[5pt]
$(b)$ $\lambda$ has strictly positive left and right eigenvectors,\\
[5pt]
$(c)$ the eigenvectors are unique up to constant multiples,\\
[5pt]
$(d)$ let $\ell, r$ be  the left and right eigenvectors for $\lambda$ then $\ell\cdot r<\infty$ if and only if $\Ab$ is positive recurrent,\\
[5pt]
$(e)$ if $0\le S\le \Ab$ and $\beta$ is the Perron value for $S$ then $\beta\le\lambda$, if $S$ is recurrent then there is equality if and only if $S=\Ab$,\\
[5pt]
$(f)$ $\displaystyle \lim_{n\rightarrow \infty}\Ab^n/\lambda^n={\bm 0}$ if $\Ab$ is null recurrent, and $\displaystyle \lim_{n\rightarrow \infty}\Ab^n/\lambda^n=r \ell$, normalized
so that $\ell \cdot r=1$, if $\Ab$ is positive recurrent.
\end{Prop}

From here, we list up the lemmas (including the lemmas of the Generalized Perron-Frobenius Theorem) necessary to prove the main theorem.

\begin{Lem}[{\cite[Lemma 7.1.6 (iv)]{Kitchens98}}]\label{Lemma 7.1.6 (iv)}
Suppose $\Ab$ is a countable, nonnegative matrix. Further, suppose it is irreducible and aperiodic. Then
\begin{align*}
H_{ii}(z)=\frac{1}{1-L_{ii}(z)}=\frac{1}{1-R_{ii}(z)}, \quad |z|<\frac{1}{\lambda}.
\end{align*}
\end{Lem}

\begin{Lem}[{\cite[Lemma 7.1.8]{Kitchens98}}]\label{Lemma 7.1.8}
Suppose $\Ab$ is a countable, nonnegative, irreducible, and aperiodic matrix. Then\\
$(i)$~ either every $H_{ij}(1/\lambda)$ is finite or every $H_{ij}(1/\lambda)$ is infinite, and \\
$(ii)$~ either every $L_{ii}(1/\lambda)$ is less than one or every $L_{ii}(1/\lambda)$ is equal to one.
\end{Lem}

We define the following vectors. For each $i, j$, define a row and column vector 
\begin{align*}
\ell^{(i)}=(L_{i0}(1/\lambda), L_{i1}(1/\lambda), \dots, L_{ij}(1/\lambda), \dots), \quad 
r^{(j)}=
\left(
\begin{array}{c}
 R_{0j}(1/\lambda) \\
 R_{1j}(1/\lambda) \\
 \vdots \\
 R_{ij}(1/\lambda) \\
  \vdots 
\end{array} 
\right).
\end{align*}

By the following Lemma, these vectors are exactly the left and right eigenvectors, respectively.

\begin{Lem}[{\cite[Lemma 7.1.9 (i)]{Kitchens98}}]\label{Lemma 7.1.9 (i)}
Let $\Ab$ be a countable, nonnegative, irreducible, aperiodic, and recurrent matrix. Then
\begin{align*}
    \ell^{(i)}\Ab=\lambda\ell^{(i)}, ~
    \Ab r^{(i)}=\lambda r^{(i)} \quad\text{for all $i$}.
\end{align*}
\end{Lem}

\begin{Prop}[Finite Approximation Theorem, {\cite[Theorem 7.1.4]{Kitchens98}}]\label{FAT}
Let $A$ be a countable, infinite, non-negative matrix. Assume $A$ is irreducible, aperiodic, and recurrent.
\begin{enumerate}
    \item If $\lambda$ is the Perron value for $A$, 
    $$
    \lambda = \sup\{\lambda(A'): A' \mbox{ is a finite, irreducible and \textcolor{red}{aperiodic} submatrix of }A \},
    $$
    where $\lambda (A') \mbox{ is the Perron value of the finite matrix } A'$.
    \item Let $\ell$ and $r$ be the left and right eigenvectors of $A$
    for the Perron value $\lambda$, normalized so that for some index $i$,
    $\ell_i=r_i=1$. Let $\{A_n\}$ be an increasing family of finite irreducible, aperiodic submatrices of $A$ that converge to $A$. 
    Let $\ell^{(n)}$, $r^{(n)}$ be the left and right eigenvectors corresponding to their Perron values, normalized so that 
    $\ell_i^{(n)}=r_i^{(n)}=1$. Then $\lim \ell_j^{(n)}=\ell_j$ and
    $\lim r_j^{(n)}=r_j$ for all $j$ in the index set.
\end{enumerate}
\end{Prop}


In the $\beta$-expansion case, 
the countable incidence matrix of the Markov diagram
given in \Cref{Markov diagram for beta-transformation} 
is
\begin{align*}
\Ab=
\left(
\begin{array}{ccccc}
b_1 & 1 & & \\
b_2 &  & 1 &  \\
b_3 & & & \ddots\\
\vdots & &  & 
\end{array} 
\right),
\end{align*}
and the left and right eigenvectors corresponding to the eigenvalue $\beta$ can be written explicitly as  
\begin{align*}
\ell^{(0)}=(1, 1/\beta, 1/\beta^2, \dots, 1/\beta^i, \dots), \quad
r^{(0)}=
\left(
\begin{array}{c}
 1 \\
 T_{\beta}(1) \\
 \vdots \\
 T_{\beta}^i(1) \\
  \vdots 
\end{array} 
\right).
\end{align*}
\textcolor{blue}{
We see that the $i$-th entries of $\ell^{(0)}$ and 
$r^{(0)}$ give the width and height of the rectangle $\mathcal R_{\beta,i}$.
}

\begin{Lem}[{\cite[Lemma 7.1.13]{Kitchens98}}]\label{Lemma 7.1.13}
Either $L'_{ii}(1/\lambda)=\sum_{n=1}^{\infty}n\ell_{ii}(n)/\lambda^{n-1}$ is finite for all i and j, or it is infinite for all i and j.
\end{Lem}

We define the mean recurrence weight $\mu(i)$ as
\begin{align*}
\mu(i)=\frac{1}{\lambda}L'_{ii}(1/\lambda).
\end{align*}

\begin{Lem}[{\cite[Lemma 7.1.21]{Kitchens98}}]\label{Lemma 7.1.21}
If $\Ab$ is positive recurrent, then $\ell^{(i)}\cdot r^{(i)}=\mu(i)$.
\end{Lem}

We define a new quantity $\varrho(i)$ as 
\begin{align*}
    \varrho(i)=\limsup{\sqrt[n]{\ell_{ii}(n)}}.
\end{align*}
Clearly, $\varrho(i)\le \lambda$. 

\begin{Lem}[{\cite[Lemma 7.1.25]{Kitchens98}}]\label{Rho}
If $\Ab$ is a matrix where $\varrho(i)<\lambda$ for some state $i$, then $\Ab$ is positive recurrent.
\end{Lem}

\section{Natural Extension}\label{Natural extension}

By \Cref{Hofbauer tower and Markov Diagram}, the incidence matrix $A_{-\beta}$ of Hofbauer's Markov Diagram in \Cref{Markov diagram}, \ref{Hofbauer tower} is given by
\begin{align}\label{A_beta}
\Ab_{-\beta}=
\bordermatrix{ & 0 & 1 & 2 & \cdots & n & n+1 & \cdots & m+a+1 & \cdots &  s+1 & \cdots \cr
0 & b_1 & 1 & & & & & & & &\cr
1 & b_1-b_2-1 & 1 & 1 & & & & & & & &\cr
\vdots & & & & \ddots & & & & & & \cr
n & i & 0/1 & & & & 1 & & & & &\cr
\vdots & & & & & &  & \ddots & & & & \cr
m+a & & & & & & & & 1 & & & \cr 
\vdots & & & & & & & & & \ddots & & \cr
s & j & & & & 1 & & & & & 1 &\cr
\vdots & & & & & & & & & & &\ddots \cr}, 
\end{align}
\textcolor{blue}{
where the $(i, j)$ entry $a_{ij}$ of the matrix $A_{-\beta}$ is the number of the edges from the state $i$
to the state $j$ for $i,j\ge0$.}

\begin{Prop}\label{Prop: irreducible aperiodic and recurrent}
\textcolor{red}{For $\beta> (1+\sqrt{5})/2$}, $\Ab_{-\beta}$ is irreducible, aperiodic, and recurrent. Furthermore, the Perron value of $\Ab_{-\beta}$ is $\beta$.
\end{Prop}

\textcolor{red}{
The same assertion holds for $1<\beta\le (1+\sqrt{5})/2$, if we properly 
restrict  
$A_{-\beta}$ to its irreducible component, reflecting the support of the absolutely continuous invariant measure of $T_{-\beta}$.
}

\begin{proof}
Since the map $T_{-\beta}$ has a unique invariant measure that is equivalent to the Lebesgue measure, for any interval $I, J \subset [\ell_\beta,r_\beta)$, there exists $n$, 
\begin{align*}
    \mu\left(T_{-\beta}^{-n}(I)\cap J\right)>0.
\end{align*}
This implies that every state can be reached from every other state. Hence $\Ab_{-\beta}$ is irreducible.
Also, since $(\Ab_{-\beta})_{00}>0$, $\Ab_{-\beta}$ is  aperiodic. 
By the definition of the map $T_{-\beta}$, 
every interval is expanded by a factor of $\beta$, and subdivided by the discontinuity points. When a subdivided interval becomes full, it is sent
back to the unit interval. Counting such intervals yields
\begin{align*}
L_{00}\left(\frac{1}{\beta}\right)=\sum_{n=1}^{\infty}\frac{\ell_{00}(n)}{\beta^n}=1.
\end{align*}
Thus, by \Cref{Lemma 7.1.6 (iv)} and \Cref{Lemma 7.1.8}, $\Ab_{-\beta}$ is recurrent, the radius of convergence of $H_{ii}(z)$ is $1/\beta$, and the Perron value of $\Ab_{-\beta}$ is $\beta$.
\end{proof}

For a finite set $\A$, we denote by $\A^*$ the monoid generated by concatenation over 
$\A$ which has the identity; namely the empty word. 

\begin{Lem}
\label{Recur1}
Let $W$ be a non-empty set of positive integers. Let 
$s_W(k)$ be the cardinality of words in $W^*$ whose arithmetic sum of letters
is equal to $k$. Then
$$
s_W(k)=\sum_{j\in W} s_W(k-j).
$$
holds for $k\ge 1$. This
does not hold at $k=0$,
since $s_W(0)=1$ by the empty word.
\end{Lem}

\begin{proof}
Let $M(k)$ be the set of words in $W^*$ whose sum of letters
is equal to $k$. 
The statement is clear by classifying $M(k)$ with respect to 
the last letter of $x\in M(k)$. 
\end{proof}

\begin{Lem}
\label{Recur2}
Let $u, m$ be positive integers with $u\le m$ and $a_u,\ldots,a_m$ are non-negative integers.  
Set
$$
S_{u,m}(n)=\sum_{u a_u+\dots + m a_m<n} \binom{a_u+a_{u+1}+\dots+ a_m}{a_u,a_{u+1},\dots,a_m}.
$$
Then we have
$$
S_{u,m}(n) = 1+\sum_{j=u}^m S_{u,m}(n-j).
$$
for $n\in \N$.
\end{Lem}

\begin{proof}
Set $W=\{u,u+1,\dots,m\}$.
Since the multinomial coefficient counts permutations of a multiset,
the sum $S_{u,m}(n)$ is 
the number of words in $W^*$ whose arithmetic sum of letters is less than $n$.
Using Lemma \ref{Recur1}, we have
$$
S_{u,m}(n)=\sum_{0\le k<n} s_W(k)
=1+\sum_{1\le k<n} s_W(k)
=1+\sum_{j\in W} \sum_{0\le k<n} s_W(k-j) =1+\sum_{j\in W} \sum_{0\le k<n-j} s_W(k).
$$
\end{proof}

\begin{Th}
\label{SystemPositiveRecurrent}
For $\beta>(1+\sqrt{5})/2$, $\Ab_{-\beta}$ is positive recurrent.
\end{Th}

The idea is to estimate $\varrho(i)$ for some $i$
and use Lemma \ref{Rho}. 
The key to this proof is a subtle property of the Markov diagram in Corollary \ref{Defect}.

\begin{proof}
Recall that $\ell_{00}(n)$
is the number of walks of length $n$
in the Markov diagram which start from $0$ and return to $0$ without visiting $0$ in the middle.
Until $n-1$ seconds, the walk will climb up one step in our diagram, or climb
down by $i-j$ steps by the arrow $i\rightarrow j$
with $i\ge j>0$. 
The defect $i-j+1$ gives the difference
from the expected increase by one second, that is, one.
At $n-1$ seconds, the walk is at the level
\begin{equation}
\label{Pos}
(n-1) -\sum \text{defects} \ge 0.
\end{equation}
Then finally, at $n$ seconds, we get back to $0$ by multiplicities
at most $b_1$.
By the construction of the Markov diagram in \Cref{How to construct the Markov diagram}, this is the only occasion when we may have plural edges.
Putting $W=\{1,2,\dots,n-1\}$, the walk of length $n-1$ is encoded as a word of $W^*$ by the order of occurrences 
of defects.
This gives a map from the set of walks of length $n-1$ starting from $0$ which never return to $0$
to the set of words over defects $\{1,2,\dots,n-1\}$ whose arithmetic sum is less 
than $n$ in view of (\ref{Pos}).
This map is injective by \Cref{Defect} since from such a word we can uniquely retrieve the walk
so that the defects occur according to the order of the word.
By this discussion, we have
$$
\ell_{00}(n)\le b_1 S_{1,n-1}(n).
$$
By using the recursive formula of \Cref{Recur2}, we have
$$
S_{1,m}(n)\le 2^{n-1}
$$
for $n,m\in \N$. This implies 
$$
\varrho(0)=\limsup_{n\to \infty}
\ell_{00}(n)^{1/n}\le 2.
$$
Thus, if $\beta>2$, $A_{-\beta}$
is positive recurrent by \Cref{Rho}.
For $\beta\le 2$, we \textcolor{red}{have $b_1=1$ and} consider
$\ell_{11}(n)$ instead.
The set of walks that start from state $1$ and return 
to state $1$ without visiting state $1$ in between can be classified into two types. The first type
of walk visits $0$ in the meantime, and the second type of walk 
never visits $0,1$ in the middle. 
The number of walks
of the second type is bounded 
by $S_{2,n-1}(n)$ in the same manner as above with the choice $W=\{2,3,\dots,n-1\}$.
Note that by \Cref{Defect}, the defect $1$ does not appear since we are interested in $\ell_{11}(n)$
for $n\ge 2$.
For the first type, the walk must stay in $0$ until
$n-1$ seconds after it reaches the
state $0$. Therefore, the number of the first 
type of walks is estimated by
$\sum_{j=3}^{n-2} S_{2,j-1}(j)$, by applying a similar estimate
at the last level $j$
before hitting the state $0$.
There is no choice of walks 
after landing the state $0$, because
$b_1=1$.
Therefore, we have
$$
\ell_{11}(n)\le S_{2,n-1}(n)+
\sum_{j=3}^{n-2} S_{2,j-1}(j).
$$
By using \Cref{Recur2}, we can show
$$
S_{2,m}(n)\le F_n
$$
for $n, m\in \N$ by induction.
Here, $F_n$ is the Fibonacci number
defined by $F_{n+2}=F_{n+1}+F_n$
with $F_0=0$ and $F_1=1$
and we use an identity:
$$
F_n=1+\sum_{j=1}^{n-2} F_{j}.
$$
This implies
$$
\ell_{11}(n)\le 2F_n, \quad
\varrho(1)=\limsup_{n\to \infty} \ell_{11}(n)^{1/n}
\le \frac{1+\sqrt{5}}{2}
$$
and $A_{-\beta}$ is positive recurrent when $(1+\sqrt{5})/2<\beta\le 2$ by \Cref{Rho}.
The proof is finished.
\end{proof}

\begin{Rem}
\textcolor{red}{
By choosing a proper irreducible component of the Markov Diagram, we can show 
positive recurrence for smaller $\beta$'s. For example, 
in the above proof, we may select $\ell_{22}(n)$
for $\beta\le 2$. Since two self-loops at $0$ and $1$ are 
joined by an arrow from state $0$ to $1$, the number of walks of the first type is bounded by
$n \sum_{j=4}^{n-2} S_{3,j-1}(j)$, while the one for
the second type is bounded by $S_{3,n-1}(n)$.
This yields the bound $\varrho(2)\le \alpha$, where $\alpha\approx 1.46557$ is the root of $x^3-x^2-1$
in view of Lemma \ref{Recur2}.
When $\beta\le (1+\sqrt{5})/2$, the state $0$ should be removed from the diagram
to ensure irreducibility, and we see that the associated countable matrix is positive recurrent for 
$\beta>\alpha$ (See \Cref{Example:x^2-x-1}). 
However, to finish the proof for all $\beta>1$,  
it still remains
a considerable combinatorial study of the system. 
We postpone this task to the next paper. }
\end{Rem}

By \Cref{GPFT}, there exists exactly one positive left eigenvector and one positive right eigenvector of the matrix $A_{-\beta}$
corresponding to the eigenvalue $\beta$ 
up to constant multiples.
Assume that $\beta > (1+\sqrt{5})/2$, then $A_{-\beta}$
is positively recurrent.
By \Cref{FAT} (Finite Approximation Theorem), these eigenvectors can be approximated using a sufficiently large number of
states of the Markov diagram. Therefore, we obtain the main theorem below, which gives a concrete 
construction of the natural extension, using the rectangles determined by these eigenvectors:
\begin{Th}\label{Th:left and right eigenvector}
Assume that $\beta > (1+\sqrt{5})/2$.
The right eigenvector of $A_{-\beta}$ 
corresponding to $\beta$ 
is made explicit as 
\begin{align}\label{right eigenvector}
    r=(|R_0|,|R_1|,\ldots,|R_n|,\ldots)^T,
\end{align}
where $|R_n|$ is the length of the $n$-th floor of the
Markov diagram in \Cref{Hofbauer tower} and
the vector $v^T$ is a transpose of a vector $v$.

Let $\ell:=\ell^{(0)}=(\ell_0,\ell_1, \ldots)$ be the left eigenvector of $A_{-\beta}$ corresponding to $\beta$ and
\textcolor{blue}{$L_n$ an interval with the length $\ell_n$
for $n \in \mathbb N \cup \{0\}$.}
We stack the rectangles $\mathcal{R}_n=R_n\times L_n$
vertically \textcolor{red}{and obtain the disjoint union of rectangles $\mathcal{H}$ (See \Cref{Fig:NE_Example1and2})
whose $2$-dimensional Lebesgue measure is finite.}
Then we can construct the natural extension of $T_{-\beta}$ on $\mathcal{H}$ \textcolor{red}{which preserves the $2$-dimensional Lebesgue measure}. Consequently, the density function
of the absolutely continuous invariant measure of $T_{-\beta}$
is
\begin{equation}
\label{Density2}    
h_{-\beta}(x)=\sum_{n\ge 0, ~x\in R_n} \ell_n .
\end{equation}
\end{Th}


\begin{proof}
Let $A_{-\beta}=(a_{ij})$.
By the construction of Hofbauer's Markov diagram
and (\ref{R_m}), (\ref{R_m+a}), (\ref{R_m+n}),
$$ 
a_{n0} \frac{|R_0|}{\beta}+
a_{n1} \frac{|R_1|}{\beta}+
\cdots+
a_{ni} \frac{|R_i|}{\beta}+\cdots=|R_n|$$
for any positive integer $n$, and we can show that
the right eigenvector corresponding to $\beta$ is given by 
(\ref{right eigenvector}).
By \Cref{Lemma 7.1.21} (or \Cref{GPFT}), $\ell^{(0)}\cdot r^{(0)}$ is finite. 
By \Cref{Lemma 7.1.9 (i)}, \Cref{right eigenvector} coincides with  
\begin{align*}
r^{(0)}=\left(R_{00}(1/\beta), R_{10}(1/\beta), \cdots ,R_{i0}(1/\beta), \cdots \right)^{T},
\end{align*}
up to constant multiples. 
So, $\ell^{(0)} \cdot r<0$,  that is, the total area of $\mathcal H$ is finite.

Now we can construct a natural extension of $(-\beta)$-transformation $\mathcal{T}_{-\beta}$ on $\mathcal H$.
The first coordinate of the natural extension is given by $T_{-\beta}$ itself. For each rectangle $\mathcal{R}_n$, we expand it by a factor of $-\beta$ in the $x$-direction and contract it by a factor of $1/\beta$ in the $y$-direction. 
Hence, the area is preserved. 

Next, we divide the resulting rectangle at 
$x=(-\beta) \cdot (r_\beta-i/\beta)$ for $i=1,2,\ldots,b_1$, 
and move each subrectangle according to the arrow of the Markov diagram, stacking them onto the corresponding floor. 
Since $\ell=(\ell_0,\ell_1,\ldots)$
is a left eigenvector of $\Ab_{-\beta}$ with eigenvector $\beta$,
$$ 
a_{0n} \frac{\ell_0}{\beta}+
a_{1n}\frac{\ell_1}{\beta}+
\cdots+ a_{in}\frac{\ell_i}{\beta}+\cdots =\ell_n.
$$

Therefore, the total height of the rectangles stacked on the
$n$-th floor is exactly equal to the height of $\mathcal{R}_n$.
Consequently, the image of the construction coincides with $\mathcal{H}$.
We now present the specific mapping $\mathcal{T}_{-\beta}$ of the natural extension of the $(-\beta)$-transformation, in the spirit of \cite{Dajani-Kraaikamp-Solomyak96}. It is based on the Markov diagram (See \Cref{Markov diagram}). 
Let $\ell_{\beta}=(b_1b_2\cdots)_{-\beta}$. 
For ease of understanding, 
we start with a construction in a symbolic setting. Formally, we 
denote the second component of $(x, y)\in \mathcal{R}_i=R_i\times L_i$ by
\begin{align*}
y=[\underbrace{0\cdots0}_{i}c_{i+1}c_{i+2}\cdots]_{-\beta}.
\end{align*}
There is an arrow from the state $i$ labeled by $d_1(x)$.
If it is an uparrow, then
\[
y^*:=[\underbrace{0\cdots0}_{i+1}c_{i+1}c_{i+2}\cdots]_{-\beta}, 
\]
if it is a downarrow or a selfloop from the state $i$ to $n$, then
\[
y^*:=[\underbrace{0\cdots0}_{n}b_{n+1}b_{n+2}\cdots b_i d_1(x) c_{i+1}c_{i+2}\cdots]_{-\beta}.
\]
Then, $\mathcal{T}_{-\beta}$ is defined by $\mathcal{T}_{-\beta} (x, ~y):=\left(T_{-\beta}(x), ~y^*\right)$. 

We wish to realize this map as a Lebesgue  
measure preserving map among rectangles (cylinders) like Baker's map. 
The rectangle is stretched $(-\beta)$ times in the horizontal 
direction and shrinked $1/\beta$ times 
in the vertical direction, and cut into pieces and sent to the next rectangle according to the arrow of the Markov diagram.
To this end, we define the value $[\omega]_{-\beta}$
of the infinite words $\omega$ 
as a point in $L_n$ by our coding.
Hereafter, we are obliged to fix the order of stacking. Let us
align the center lines of the 
rectangles ${\mathcal R}_n$ to the real axis
and stack $a_{in}$ rectangles of height $\ell_i/\beta$ and width $r_n=|R_n|$ 
as the contribution of the incoming arrows from state $i$ to state $n$, 
keeping the order\footnote{For a fixed $R_n$, we may stack by any order of the incoming edges
to $R_n$ for this construction. 
The simplest choice
is to stack them 
in the order of $i$, but we may not expect a good embedding in $\R^2$ as 
in Example \ref{Example:x^2-x-1}.}
of~$\prec$. More precisely,
if $i$ is odd, we stack upward in the lexicographic order of $(i,d_1(x))$,
and if $i$ is even, then
we stack downward, in the negative direction, in the order of~$(i,d_1(x))$. 

Tracing back the itinerary in the Markov diagram, 
we know the exact
location of the interval in $L_n$ which corresponds 
to the cylinder determined by a prefix of $\omega$. 
Taking longer prefixes, we obtain a sequence of 
nested intervals which converges to a single point, 
as the prefix length tends to infinity. 
In this way, we can define the height of $\omega$ from the center line, 
i.e., the
exact value of $[\omega]_{-\beta}$
in $L_n$. This map $\omega\to [\omega]_{-\beta}$ may not be realized as
a simple interval map. See Example \ref{Example:x^2-x-1} for a lucky case.

\end{proof}

\textcolor{red}{
After normalization to the probability measure, \Cref{Th:left and right eigenvector}
gives the expected return time to a given cylinder (c.f. a special case of the renewal theorem \cite[Theorem 7.1.18]{Kitchens98}),
as the reciprocal of the area of the rectangle 
(c.f. \cite[Lemma 7.1.21]{Kitchens98}). 
Therefore, we find that renewal theory for countable Markov shift, and the Ka\u{c} formula (c.f. Theorem 1.2.2 in \cite{VianaOlivieira}) consistently work in this setting (c.f. \cite[Lemma 7.1.22]{Kitchens98}).
From our natural extension, we obtain 
a comprehensive
understanding of the return time structure of the dynamical system of $T_{-\beta}$.}

We attempted to write the dual part $y$ as $(-\beta)$-expansion, following an analogy to (\ref{NaturalBeta})
in \cite{Dajani-Kraaikamp-Solomyak96}.
However, the dual map does not seem to be realized as an easy interval map, except for the golden ratio. See \Cref{Example:x^2-x-1}.

Since $T_{-\beta}$ is ergodic
with respect to 
a unique absolutely continuous 
invariant measure, the density  (\ref{Density}) and 
(\ref{Density2}) must coincide up to a constant multiple.
We do not have a direct combinatorial proof of this fact. 
\textcolor{blue}{In some cases, we can make the left eigenvector explicit using this coincidence.}

\begin{Prop}\label{ell=(c...)}
\textcolor{red}{
If $\beta>(1+\sqrt{5})/2$ and 
$T^{2m-1}(\ell_{\beta})\notin I_{b_1}$ for $m\in \N$,} then the left eigenvector is
$$\left(\frac{\beta^2-\beta-1}{\beta^2-1},\frac{1}{\beta}, \frac{1}{\beta^2},\dots\right).$$
\end{Prop}

\begin{proof}
Under the assumption, we have
\begin{align*}
    R_n=
\begin{split}
\left\{
\begin{array}{ll}
\displaystyle \left[\ell_{\beta},~ T_{-\beta}^n(\ell_{\beta})\right] \quad n: \text{odd}, \\
[10pt]
\displaystyle \left[T_{-\beta}^n(\ell_{\beta}),~ r_{\beta}\right] \quad n: \text{even}.
\end{array} \right.
\end{split}
\end{align*}
Therefore, we have
\begin{align*}
    \sum_{n\ge 0,~T^n(\ell_{\beta}^+)\le x}\frac{1}{(-\beta)^n}
    &=1+\sum_{T^n(\ell_{\beta}^+)\le x,~n:odd}\frac{1}{(-\beta)^n}
    +\sum_{T^n(\ell_{\beta}^+)\le x,~n:even}\frac{1}{(-\beta)^n}\\
    &=1-\sum_{n:odd}\frac{1}{\beta^n}+\sum_{T^n(\ell_{\beta}^+)> x,~n:odd}\frac{1}{\beta^n}+\sum_{T^n(\ell_{\beta}^+)\le x,~n:even}\frac{1}{\beta^n}\\
    &=1-\sum_{n:odd}\frac{1}{\beta^n}+\sum_{n\ge 1,~x\in R_n}\frac{1}{\beta^n}.
\end{align*}
\end{proof}
See Table \ref{LeftEigenvectors} for examples of Proposition \ref{ell=(c...)}.




\section{Examples}\label{Examples}
When $d^*(\ell_{\beta},-\beta)$ is eventually periodic, 
the Markov diagram can be represented by a finite matrix. This is done by the future cover (c.f. \cite[Exer. 3.2.8]{Lind-Marcus:95}), i.e., we identify states by their future. It is an equivalence relation among the states of the Markov
diagram induced by the set of 
all possible infinite words starting from a given state. 
We denote the result of this identification by $\Ab^*_{-\beta}$.
\begin{Ex}[$\underline{x^2-x-1=0}$; $\beta=(1+\sqrt{5})/2=1.618\dots$, $d(\ell_\beta,-\beta)=1\overline{0}$]\label{Example:x^2-x-1}
\begin{align*}
\Ab_{-\beta}=
\left(
\begin{array}{c|cccccc}
1 & 1 &  & &  & & \\ \hline
 & 1 & 1 & & & & \\
 & 0 & & 1 & & &  \\
 & 1 &  & & 1 & & \\
 & 0 & & & & 1 & \\
 & 1 & & & & & \ddots \\
 & \vdots& & & & & \\
\end{array} 
\right), \quad
\Ab^*_{-\beta}=
\left(
\begin{array}{ccc}
1 & 1 &  \\
  &   & 1 \\
1 & 1 & 
\end{array} 
\right).
\end{align*}
\end{Ex}

In this Markov Diagram, once the orbit reaches state $1$, it never returns to state $0$,
\textcolor{red}{and it must be removed for irreducibility.} 
We can use the left eigenvector \textcolor{red}{$(1/\beta,1/\beta^2,\dots)$ which can be read off from Proposition \ref{ell=(c...)}.}
By the identification induced by the future cover, we obtain the left eigenvector $(1/\beta,1/\beta,1/\beta^2)$ of $A^*_{-\beta}$.
It follows from the proof of \Cref{Th:left and right eigenvector} that we have
\[
y^*=
\begin{cases}
[0^{n}c_1c_2\cdots]_{-\beta},
& d_1(x)=0,\\
[0^{n-1}1c_1c_2\cdots]_{-\beta},
& d_1(x)=1 ~(n\text{ is odd}),
\end{cases}
\]
for $(x,y)$ in $\mathcal{H}\setminus \mathcal{R}_{0}=\bigcup_{n\ge 1}\mathcal{R}_n$. 
Consider the associated one-sided shift 
$X_{\bf dual}$, i.e., the closure of all infinite words
associated with these dual coordinates.
For $\beta=(1+\sqrt5)/2$, its forbidden words coincide 
with 
$\{10^{2k-1}1\}_{k\ge 1}$. 
Thus, if $f: \Sigma_{\bf dual}\to [\ell_{\beta}, r_\beta]$ denotes the representation map as a power series in base $(-\beta)$, this coincides 
with the inverse of the $(-\beta)$-expansion map up to countable exceptions,
i.e., $f$ fails to be injective only on the countable set $\bigcup_{n\ge1}T_{-\beta}^{-n}(r_\beta)$.
In this manner, in each $\mathcal{R}_i$ the value is realized as the $(-\beta)$-expansion.
One can also realize this 
natural extension in $\R^2$ in the following simple form:
\begin{align*}
\mathcal{T}_{-\beta}:(x,y) \mapsto 
\begin{cases}
    \left(T_{-\beta}(x), \displaystyle -\frac{1+y}{\beta}\right), & (x,y)\in\mathcal{X}_0\cup\mathcal{X}_2, \\
    \left(T_{-\beta}(x), \displaystyle -\frac{y}{\beta}\right), & (x,y)\in\mathcal{X}_1, \\
\end{cases}
\end{align*}
which is defined on $\mathcal{X}=\mathcal{X}_0\cup \mathcal{X}_1 \cup \mathcal{X}_2\subset [\ell_{\beta}, r_{\beta}]^2$
with 
$\mathcal{X}_0=[\ell_{\beta},0]\times[\ell_{\beta},0],~
\mathcal{X}_1=[0, r_{\beta}]\times[\ell_{\beta},0],~
\mathcal{X}_2=[\ell_{\beta},0]\times[0, r_{\beta}]$.
See \Cref{fig:NE_Example_3}. 

\begin{figure}[h]
    \centering
    \begin{tikzpicture}[
  x=5.2cm, y=5.2cm, >=Latex,
  line/.style={draw=black, thick},
  lab/.style={font=\small}]
\pgfmathsetmacro{\b}{(1+sqrt(5))/2}      
\pgfmathsetmacro{\lb}{-\b/(\b+1)}        
\pgfmathsetmacro{\rb}{1/(\b+1)}          
\pgfmathsetmacro{\c}{\rb-1/\b}           

\begin{scope}
  
  \draw[line] (0,\lb) -- (0,\rb);
  \draw[line] (\lb,0) -- (\rb,0);

  \shade[shading=axis, left color=green!5, right color=green!50, shading angle=45]
  (\lb,\lb) rectangle (\c,0); 
  
   \shade[shading=axis, left color=green!5, right color=green!50, shading angle=45]
  (\c,\lb) rectangle (0,0); 

  \shade[shading=axis, left color=red!5, right color=red!50, shading angle=45]
  (0,\lb) rectangle (\rb,0); 
  
  \shade[shading=axis, left color=yellow!5, right color=yellow!50, shading angle=45]
  (\c,0) rectangle (0,\rb); 

  \shade[shading=axis, left color=yellow!5, right color=yellow!50, shading angle=45]
  (\lb,0) rectangle (\c,\rb); 

  \draw[line] (\lb,\lb) -- (\lb,\rb) -- (0,\rb) -- (0,0) -- (\rb,0) -- (\rb,\lb) -- (\lb,\lb);
  \draw[line] (0,\lb) -- (0,0);
  \draw[line] (\lb,0) -- (0,0);
  \draw[densely dashed, gray, thick] (\c,\lb) -- (\c,\rb);

  \node[lab] at ({(\lb+0)/2},{(\lb+0)/2}) {$\mathcal{X}_0$};
  \node[lab] at ({(0+\rb)/2},{(\lb+0)/2}) {$\mathcal{X}_1$};
  \node[lab] at ({(\lb+0)/2},{(0+\rb)/2}) {$\mathcal{X}_2$};
\end{scope}

\draw[->, thick] (\rb+0.1, \lb/2+\rb/2) -- (\rb+0.5,\lb/2+\rb/2);
\node[lab, above] at (\rb+0.3,\lb/2+\rb/2) {$\mathcal{T}_{-\beta}$};

\begin{scope}[shift={(\rb+1.2,0)}]

  \shade[shading=axis, left color=green!50, right color=green!5, shading angle=30]
  (\lb,-1/\b^2) rectangle (0,0); 

  \shade[shading=axis, left color=green!50, right color=green!5, shading angle=30]
  (0,0) rectangle (\rb,-1/\b^2); 
  
  \shade[shading=axis, left color=yellow!50, right color=yellow!5, shading angle=30]
  (\lb,\lb) rectangle (0,-1/\b^2); 
  
  \shade[shading=axis, left color=yellow!50, right color=yellow!5, shading angle=30]
  (0,\lb) rectangle (\rb,-1/\b^2); 

  \shade[shading=axis, left color=red!50, right color=red!5, shading angle=30]
  (\lb,0) rectangle (0,\rb); 

  \draw[line] (\lb,\lb) -- (\lb,\rb) -- (0,\rb) -- (0,0) -- (\rb,0) -- (\rb,\lb) -- (\lb,\lb);
  \draw[line] (\lb,0) -- (\rb,0);
  \draw[line] (\lb,-1/\b^2) -- (\rb,-1/\b^2);

  \node[lab] at ({(\lb+0)/2},{(0+\rb)/2}) {$\mathcal{T}_{-\beta}(\mathcal{X}_1)$};
  \node[lab] at ({(\lb+\rb)/2},{(-1/\b^2+0)/2}) {$\mathcal{T}_{-\beta}(\mathcal{X}_0)$};
  \node[lab] at ({(\lb+\rb)/2},{(\lb-1/\b^2)/2}) {$\mathcal{T}_{-\beta}(\mathcal{X}_2)$};
\end{scope}
\end{tikzpicture}
    \caption{The natural extension of the $(-\beta)$-transformation when $\beta=(1+\sqrt5)/2$. (The dashed line is $x=r_{\beta}-1/\beta$.)}
    \label{fig:NE_Example_3}
\end{figure}
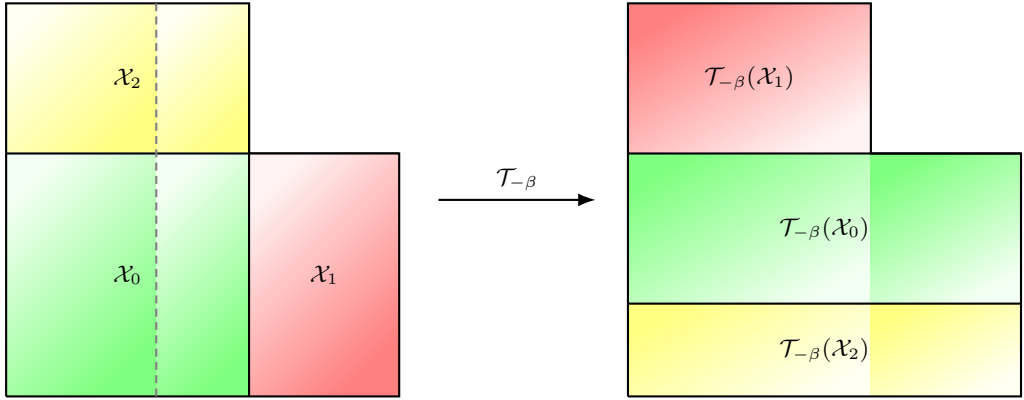

\begin{Ex}[$\underline{x^2-kx-1=0}$; $d(\ell_\beta,-\beta)=k\overline{k-1}$]\label{Example:x^2-kx-1}
\begin{align*}
\Ab^*_{-\beta}=
\left(
\begin{array}{ccc}
k & 1 &  \\
  & 1 & 1 \\
k-1 & 1 & 
\end{array} 
\right).
\end{align*}
\end{Ex}
We can see that $\left((k-1)/k,~1/k,~1/(k\beta)\right)$ and $\left(1,~1/\beta,~1-1/\beta\right)^{T}$ are, respectively, a left eigenvector and a right eigenvector of $A^*_{-\beta}$. Thus, as in the example for the golden ratio, we can explicitly describe the translation $\alpha_{x,y}$ in the second coordinate of the natural extension $\mathcal{T}_{-\beta}(x,y)= 
(T_{-\beta}(x), \displaystyle -y/\beta+\alpha_{x,y})$ on the set $\mathcal{X}=\bigcup\mathcal X_i$ determined by these eigenvectors. Here, $\alpha_{x,y}$ is a piecewise constant term determined by the pair
$(\mathcal{X}_i,d_1(x))$ such that $(x,y)\in\mathcal{X}_i$. In this realization of the natural extension on $\mathcal{X}$, we want to determine the backward orbit on the dual side only by the value of $y$. 
Note that, in such a realization, the alternating order in the symbolic representations of $y^*$ no longer reflects the numerical order of the corresponding dual values.

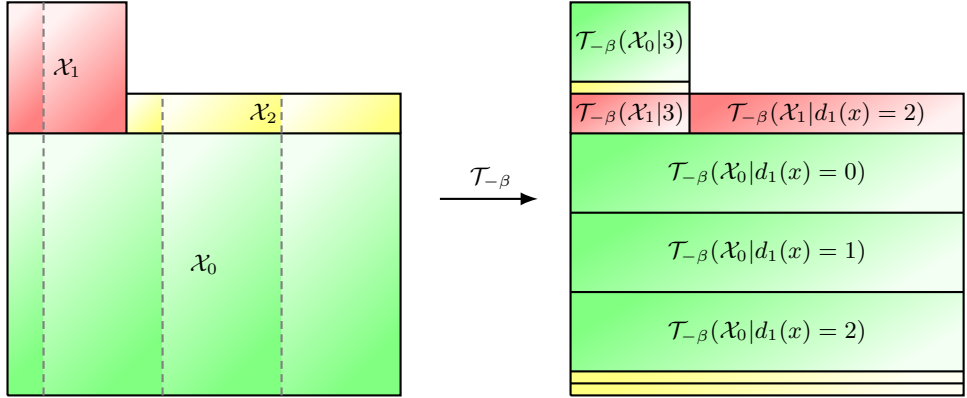
\begin{figure}[h]
    \centering
    \begin{tikzpicture}[
  x=5.2cm, y=5.2cm, >=Latex,
  line/.style={draw=black, thick},
  lab/.style={font=\small}]
\pgfmathsetmacro{\b}{(3+sqrt(13))/2}      
\pgfmathsetmacro{\lb}{-\b/(\b+1)}        
\pgfmathsetmacro{\rb}{1/(\b+1)}          

\pgfmathsetmacro{\t}{-\b*\lb-3}           

\pgfmathsetmacro{\k}{\rb-2/(3*\b)}   
\pgfmathsetmacro{\kk}{\rb-2/(3*\b)-1/(3*\b^2)}   
\pgfmathsetmacro{\kkk}{\lb+2/(3*\b^2)+4/(3*\b)}   
\pgfmathsetmacro{\kkkk}{\lb+2/(3*\b^2)+2/(3*\b)}   
\pgfmathsetmacro{\kkkkk}{\lb+2/(3*\b^2)}   
\pgfmathsetmacro{\kkkkkk}{\lb+1/(3*\b^2)}   

\pgfmathsetmacro{\c}{\rb-1/\b}           
\pgfmathsetmacro{\cc}{\rb-2/\b}          
\pgfmathsetmacro{\ccc}{\rb-3/\b}         

\begin{scope}

  \shade[shading=axis, left color=green!5, right color=green!50, shading angle=45]
  (\lb,\lb) rectangle (\ccc,\lb+2/3); 

  \shade[shading=axis, left color=green!5, right color=green!50, shading angle=45]
  (\ccc,\lb) rectangle (\cc,\lb+2/3); 

  \shade[shading=axis, left color=green!5, right color=green!50, shading angle=45]
  (\cc,\lb) rectangle (\c,\lb+2/3); 

   \shade[shading=axis, left color=green!5, right color=green!50, shading angle=45]
  (\c,\lb) rectangle (\rb,\lb+2/3); 

  \shade[shading=axis, left color=red!5, right color=red!50, shading angle=45]
  (\lb,\lb+2/3) rectangle (\ccc,\rb); 

  \shade[shading=axis, left color=red!5, right color=red!50, shading angle=45]
  (\ccc,\lb+2/3) rectangle (\t,\rb); 
  
  \shade[shading=axis, left color=yellow!5, right color=yellow!50, shading angle=45]
  (\t,\lb+2/3) rectangle (\cc,\kk); 

  \shade[shading=axis, left color=yellow!5, right color=yellow!50, shading angle=45]
  (\cc,\lb+2/3) rectangle (\c,\kk); 

  \shade[shading=axis, left color=yellow!5, right color=yellow!50, shading angle=45]
  (\c,\lb+2/3) rectangle (\rb, \kk); 

  \draw[line] (\lb,\lb) -- (\lb,\lb+2/3) -- (\rb,\lb+2/3) -- (\rb,\lb) -- (\lb,\lb); 
  \draw[line] (\lb,\lb+2/3) -- (\lb,\rb) -- (\t,\rb) -- (\t,\lb+2/3); 
  \draw[line] (\t,\kk) -- (\rb,\kk) -- 
  (\rb,\lb+2/3);  
  
  \draw[densely dashed, gray, thick] (\c,\lb) -- (\c,{\lb+2/3});
  \draw[densely dashed, gray, thick] (\cc,\lb) -- (\cc,{\lb+2/3});
  \draw[densely dashed, gray, thick] (\ccc,\lb) -- (\ccc,\rb);

  \draw[densely dashed, gray, thick] (\c,\lb+2/3) -- (\c,\kk);
  \draw[densely dashed, gray, thick] (\cc,\lb+2/3) -- (\cc,\kk);

  \node[lab] at ({(\lb+\rb)/2},{(\lb+\lb+2/3)/2}) {$\mathcal{X}_0$};
  \node[lab] at ({(\lb+\t)/2},{(\lb+2/3+\rb)/2}) {$\mathcal{X}_1$};
  \node[lab] at ({(\t+\rb)/2},{(\lb+2/3+\kk)/2}) {$\mathcal{X}_2$};
\end{scope}

\draw[->, thick] (\rb+0.1, \lb/2+\rb/2) -- (\rb+0.35,\lb/2+\rb/2);
\node[lab, above] at (\rb+0.23,\lb/2+\rb/2) {$\mathcal{T}_{-\beta}$};

\begin{scope}[shift={(\rb+1.2,0)}]


  \shade[shading=axis, left color=red!50, right color=red!5, shading angle=30]
  (\lb,\lb+2/3) rectangle (\t,\kk); 

  \shade[shading=axis, left color=red!50, right color=red!5, shading angle=30]
  (\t,\lb+2/3) rectangle (\rb,\kk); 

  \shade[shading=axis, left color=green!50, right color=green!5, shading angle=30]
  (\lb,\k) rectangle (\t,\rb); 
  
   \shade[shading=axis, left color=yellow!50, right color=yellow!5, shading angle=30]
  (\lb,\kk) rectangle (\t,\k); 

  \shade[shading=axis, left color=green!50, right color=green!5, shading angle=30]
   (\lb,\kkk) rectangle (\rb,\lb+2/3);

  \shade[shading=axis, left color=green!50, right color=green!5, shading angle=30]
  (\lb,\kkkk) rectangle (\rb,\kkk); 

  \shade[shading=axis, left color=green!50, right color=green!5, shading angle=30]
  (\lb,\kkkkk) rectangle (\rb,\kkkk); 

  \shade[shading=axis, left color=yellow!50, right color=yellow!5, shading angle=30]
  (\lb,\kkkkkk) rectangle (\rb,\kkkkk); 

  \shade[shading=axis, left color=yellow!50, right color=yellow!5, shading angle=30]
  (\lb,\lb) rectangle (\rb,\kkkkkk); 

  \draw[line] (\lb,\lb) -- (\lb,\lb+2/3) -- (\rb,\lb+2/3) -- (\rb,\lb) -- (\lb,\lb); 
  \draw[line] (\lb,\lb+2/3) -- (\lb,\rb) -- (\t,\rb) -- (\t,\lb+2/3); 
  \draw[line] (\t,\kk) -- (\rb,\kk) -- 
  (\rb,\lb+2/3);  
  
  \draw[line] (\lb,\k) -- (\t,\k);
  \draw[line] (\lb,\kk) -- (\t,\kk);
  \draw[line] (\lb,\kkk) -- (\rb,\kkk);
  \draw[line] (\lb,\kkkk) -- (\rb,\kkkk);
  \draw[line] (\lb,\kkkkk) -- (\rb,\kkkkk);
  \draw[line] (\lb,\kkkkkk) -- (\rb,\kkkkkk);

  \node[lab] at ({(\t+\rb)/2},{(\lb+2/3+\kk)/2}) {$\mathcal{T}_{-\beta}(\mathcal{X}_1|d_1(x)=2)$};
  \node[lab] at ({(\lb+\t)/2},{(\lb+2/3+\kk)/2}) {$\mathcal{T}_{-\beta}(\mathcal{X}_1|3)$};
  \node[lab] at ({(\lb+\t)/2},{(\rb+\k)/2}) {$\mathcal{T}_{-\beta}(\mathcal{X}_0|3)$};
  \node[lab] at ({(\lb+\rb)/2},{(\kkk+\lb+2/3)/2}) {$\mathcal{T}_{-\beta}(\mathcal{X}_0|d_1(x)=0)$};
  \node[lab] at ({(\lb+\rb)/2},{(\kkkk+\kkk)/2}) {$\mathcal{T}_{-\beta}(\mathcal{X}_0|d_1(x)=1)$};
  \node[lab] at ({(\lb+\rb)/2},{(\kkkkk+\kkkk)/2}) {$\mathcal{T}_{-\beta}(\mathcal{X}_0|d_1(x)=2)$};
  
\end{scope}
\end{tikzpicture}
    \caption{The natural extension of the $(-\beta)$-transformation when $k=3$. (The dashed lines are $x=r_{\beta}-1/\beta$, $x=r_{\beta}-2/\beta$ and $x=r_{\beta}-3/\beta$.)}
    \label{fig:NE_Example_4}
\end{figure}

\begin{Ex}[$\underline{x^n-x^{n-1}\cdots-x-1=0}$]\label{Example:x^n-x^{n-1}...-x-1}
\end{Ex}
This example is discussed in \cite{Kalle14}. In that paper, Kalle shows that, for the multinacci numbers $\beta_n$, the transformations $T_{\beta}$ and $T_{-\beta}$ are measurably isomorphic, and for $\beta_{n-1}<\beta<\beta_{n}$, they are not measurably isomorphic. The corresponding
natural extensions
in this case are shown in Figures~7 and~8 in \cite{Kalle14}.

\begin{Ex}[$\underline{x^3-3x^2+x-1=0}$; $\beta=2.77\dots$, $d(\ell_\beta, -\beta)=\overline{201}$
]\label{Example:x^3-3x^2+x-1}
\end{Ex}
Because $d(\ell_\beta, -\beta)$ is purely periodic with odd period, to get the correct Markov diagram from \Cref{How to construct the Markov diagram}, we have to use $d^*(\ell_\beta, -\beta)=\overline{2000}$ (using \Cref{OddCycle}).
\begin{align*}
\Ab^{*}_{\overline{201}}=
\left(
\begin{array}{ccc|cccccc}
\SB{2} & \SB{1} &  &&&&&&\\
\SB{1} & \SB{1} & \SB{1} &&&&&&\\
\SB{1} &  && 1 &&&&& \\ \hline
 &&&& 1 &&&& \\
 &&&&& 1 &&& \\
 &&&&&& 1 && \\
 &&&&&&& 1 & \\
 &&&&&&&& 1  \\
 &&& 1&&&&&\\
\end{array} 
\right), \quad 
\Ab^{*}_{\overline{2000}}=
\left(
\begin{array}{ccccccc}
2 & 1 & & & & & \\
1 & 1 & 1 & & & & \\
  & & & 1 & & & \\
1 & 1 & & & 1 & & \\
2 & & & & & 1 & \\
1 & 1 & & & & & 1 \\
  & & & 1 & & & 
\end{array} 
\right).
\end{align*}

The characteristic polynomials of $\Ab^{*}_{\overline{201}}$ and $\Ab^{*}_{\overline{2000}}$ are
\begin{align*}
&(x-1)(x+1)(x^2-x+1)(x^2+x+1)\SB{(x^3-3x^2+x-1)},\\ 
&-x^4(x^3-3x^2+x-1),
\end{align*}
respectively. 
\textcolor{red}{
They have the same Perron-Frobenius eigenvalue, but the eigenvector of 
$\Ab^{*}_{\overline{201}}$ is not suitable for our construction.}

In the case where $d(\ell_{\beta},-\beta)$ is periodic, we can obtain the left eigenvector explicitly by solving a system of linear equations, see \Cref{LeftEigenvectors}.
\begin{table}[htbp]
\centering
\begin{tabular}{|c||c|c|}\hline
Example & equation & $\ell=(\ell_0, \ell_1, \ell_2,\dots), c=\frac{\beta^2-\beta-1}{\beta^2-1}$
\\ \hline\hline
\ref{Example:x^2-x-1} & $x^2-x-1=0$ &$(0,1/\beta, 1/\beta^2,\dots)$\\\hline
\ref{Example:x^2-kx-1}  & $x^2-kx-1=0$ &\SR{$(c,1/\beta, 1/\beta^2,\dots)$}\\\hline
\multirow{2}{*}{\ref{Example:x^n-x^{n-1}...-x-1}}
& $x^n-x^{n-1}\cdots-x-1=0, (n:even)$
& \SR{$(c,1/\beta, 1/\beta^2,\dots)$}
\\ \cline{2-3}
& $x^n-x^{n-1}\cdots-x-1=0, (n:odd)$ 
&
\begin{tabular}{c}
$\displaystyle (1,a,b,b/\beta,b/\beta^2,\dots)$, \\ 
where $a=L_{01}(1/\beta), b=L_{02}(1/\beta)$
\end{tabular}
\\ \hline
\ref{Example:x^3-3x^2+x-1} & $x^3-3x^2+x-1=0$&\SR{$(c,1/\beta, 1/\beta^2,\dots)$}\\\hline
\end{tabular}
\caption{The left eigenvector when $d(\ell_{\beta},-\beta)$ is periodic.}
\label{LeftEigenvectors}
\end{table}

Finally, we explain the non-periodic case using \Cref{Example:1^k0^k} and \ref{Example:1 to 101 and 0 to 1}. In \Cref{Fig:NE_Example1and2}, the arrows are determined by the Markov diagram, whose incidence matrix is given below:
\begin{align*}
\Ab_{-\beta}=
\left(
\begin{array}{ccccccccccccc}
1 & 1 & 0 & 0 & 0 & 0 & 0 & 0 & 0 & 0 & 0 & 0 \\
0 & 1 & 1 & 0 & 0 & 0 & 0 & 0 & 0 & 0 & 0 & 0 \\
1 & 0 & 0 & 1 & 0 & 0 & 0 & 0 & 0 & 0 & 0 & 0 \\
0 & 0 & 0 & 0 & 1 & 0 & 0 & 0 & 0 & 0 & 0 & 0 \\
0 & 0 & 0 & 0 & 0 & 1 & 0 & 0 & 0 & 0 & 0 & 0 \\
0 & 0 & 0 & 1 & 0 & 0 & 1 & 0 & 0 & 0 & 0 & 0 \\ 
1 & 0 & 0 & 0 & 0 & 0 & 0 & 1 & 0 & 0 & 0 & 0 \\
0 & 0 & 0 & 0 & 0 & 0 & 0 & 0 & 1 & 0 & 0 & 0 \\
0 & 0 & 1 & 0 & 0 & 0 & 0 & 0 & 0 & 1 & 0 & 0 \\
0 & 1 & 0 & 0 & 0 & 0 & 0 & 0 & 0 & 0 & 1 & 0 \\
0 & 0 & 0 & 0 & 0 & 0 & 0 & 0 & 0 & 0 & 0 & 1 \\
0 & 1 & 0 & 0 & 0 & 0 & 0 & 0 & 0 & 0 & 0 & 0 \\
  &   &   &   &   &   &   &   &   &   &   &   & \ddots
\end{array} 
\right)
\end{align*}
and
\begin{align*}
\Ab_{-\beta}=
\left(
\begin{array}{ccccccccccccc}
1 & 1 & 0 & 0 & 0 & 0 & 0 & 0 & 0 & 0 & 0 & 0 \\
0 & 1 & 1 & 0 & 0 & 0 & 0 & 0 & 0 & 0 & 0 & 0 \\
1 & 0 & 0 & 1 & 0 & 0 & 0 & 0 & 0 & 0 & 0 & 0 \\
0 & 0 & 0 & 0 & 1 & 0 & 0 & 0 & 0 & 0 & 0 & 0 \\
0 & 0 & 1 & 0 & 0 & 1 & 0 & 0 & 0 & 0 & 0 & 0 \\
0 & 1 & 0 & 0 & 0 & 0 & 1 & 0 & 0 & 0 & 0 & 0 \\ 
1 & 0 & 0 & 0 & 0 & 0 & 0 & 1 & 0 & 0 & 0 & 0 \\
0 & 0 & 0 & 0 & 0 & 0 & 0 & 0 & 1 & 0 & 0 & 0 \\
0 & 0 & 0 & 0 & 0 & 0 & 0 & 0 & 0 & 1 & 0 & 0 \\
0 & 0 & 0 & 0 & 0 & 0 & 0 & 0 & 0 & 0 & 1 & 0 \\
0 & 0 & 0 & 0 & 0 & 0 & 0 & 0 & 0 & 0 & 0 & 1 \\
0 & 0 & 0 & 0 & 0 & 1 & 0 & 0 & 0 & 0 & 0 & 0 \\
  &   &   &   &   &   &   &   &   &   &   &   & \ddots
\end{array} 
\right).
\end{align*}

By \Cref{Th:left and right eigenvector}, we know the closed formula for horizontal lengths of the rectangles. 
We computed the vertical thicknesses by using a truncation to the first 12 states
in the figures.

Since our natural extension visualizes the return time to the cylinder, when an arrow lands on a floor, the corresponding floor becomes slightly thicker.
This phenomenon can be observed in both figures. 
In \Cref{Example:1^k0^k}, there is an arrow from the $5$-th floor to the $3$-rd floor. As a result, the $3$-rd floor is slightly thicker than the corresponding floor in \Cref{Example:1 to 101 and 0 to 1}. 
Similarly, in \Cref{Example:1 to 101 and 0 to 1}, there is an arrow from the $4$-th floor to the $2$-nd floor, and then the $2$-nd floor is slightly thicker than the corresponding floor in \Cref{Example:1^k0^k}.

To facilitate comparison of differences in thickness, the bases of the rectangles in each example are aligned
in \Cref{Fig:NE_Example1and2}. As a consequence, the \textcolor{blue}{vertical} gaps between the rectangles are not uniform.
In \Cref{Fig:NE_Example1and2}, the $9$-th floor on the right appears to touch the dashed line, but it lies entirely to the left of the dashed line.
One may \textcolor{blue}{wonder whether the $1$-st floor is always} thicker than the $0$-th floor for non-periodic cases. However, this phenomenon does not occur when arrows from the first few floors land on the $0$-th floor.

\medskip
{\bf Acknowledgements.}
\medskip

For the description of the Markov diagram in our setting, we are indebted to a detailed note by Ken'ichiro Yamamoto. 

\newpage
\bibliographystyle{siam}  
\bibliography{lit}
\end{document}